\documentclass[12pt, reqno]{amsart}

\usepackage{amssymb}
\usepackage{amsmath}
\usepackage{amscd}
\usepackage{cancel}
\usepackage{float}
\usepackage{graphicx}
\usepackage{amsfonts}
\usepackage{color}
\usepackage{xypic}

\usepackage{amsfonts}
\usepackage[cmtip,arrow]{xy}
\usepackage{cancel}
\usepackage[normalem]{ulem}
\newtheorem{theorem}{Theorem}
\newtheorem{corollary}{Corollary}
\newtheorem{lemma}{Lemma}
\newtheorem{proposition}{Proposition}
\newtheorem{definition}{Definition}
\newtheorem{remark}{Remark}
\newtheorem{example}{Example}

\newtheorem{notation}{Notation}
\newcommand{\SB}{S\!B_n}
\newcommand{\TB}{T\!B_n}
\newcommand{\TSB}{T\!S\!B_n}

\newcommand{\E}{\mathcal{E}}
\newcommand{\F}{\mathcal{F}}
\newcommand{\Psf}{\mathsf{P}_n}
\newcommand{\Psk}{\mathsf{P}_k}
\newcommand{\PP}{\mathsf{P}}
\newcommand{\A}{{\rm{A}}}
\newcommand{\B}{{\rm{B}}}
\newcommand{\C}{{\rm{C}}}
\newcommand{\D}{{\rm{D}}}
\newcommand{\Sing}{{\rm{S}}}
\newcommand{\M}{{\rm{M}}}

\newcommand{\LLL}{{\rm{L}}}

\newcommand{\N}{{\rm{N}}}

\def\a{\mathsf{a}\,}
\def\b{\mathsf{b}\,}
\def\c{\mathsf{c}\,}
\def\d{\mathsf{d}\,}
\def\f{\mathsf{f}\,}
\def\u{\mathsf{u}\,}
\def\v{\mathsf{v}\,}
\def\w{\mathsf{w}\,}
\def\x{\mathsf{x}}
\def\y{\mathsf{y}\,}
\def\I{{\rm{I}}}
\def\1{{\bf 1}}

\begin{document}

\title{Tied links and  invariants for singular links}

\author{F. Aicardi}
 \address{Sistiana 56, 34011 Trieste, Italy}
 \email{francescaicardi22@gmail.com}
 \author{J. Juyumaya}
 \address{IMUV,
 Gran Breta\~{n}a 1111, Valpara\'{i}so 2340000, Chile.}
 \email{juyumaya@gmail.com}

\thanks{The second   author was  in part supported   by   Fondecyt   Nº 1180036, Nº 1210011.}


\date{}

\begin{abstract}
Tied links and the  tied braid monoid were introduced recently by the authors and used to define new invariants for classical links. Here, we give  a version purely algebraic--combinatoric  of  tied links. With this new version    we prove that the tied braid monoid has a decomposition  like   a semi--direct   group  product. By using this decomposition  we reprove the Alexander and Markov theorem for tied links; also, we introduce the tied singular knots, the tied singular braid monoid  and certain families of Homflypt type invariants for tied singular links; these invariants are  five--variables polynomials. Finally,  we study the behavior of these invariants; in particular, we show that our invariants distinguish non isotopic singular links indistinguishable by the  Paris--Rabenda invariant.

\end{abstract}
\begin{center}
\maketitle

\begin{minipage}{10cm}
\tableofcontents
\end{minipage}
\end{center}
\section{Introduction}

Tied links and their algebraic counterpart, the tied braid monoid, were introduced by the authors in \cite{aijuJKTR1}. A tied link is a classical link    admitting ties among its components; the tied braid monoid is defined through a presentation with usual braid generators together with ties generators and defining relations coming from
the so--called bt--algebra \cite{ajICTP1}, cf. \cite{aijuJKTR1, rhJAC, maIMRN, jaPoJLMS}.
\smallbreak
Tied links contains the classical links, so every invariant for tied links defines also an invariant for classical links.
We have  constructed two invariants for tied links: the one of type Homflypt polynomial \cite{aijuJKTR1} and the other one of type Kauffman polynomial \cite{aijuMathZ}. These invariants turn out  to be more powerful, respectively, than the Homflypt and the Kauffman polynomials; therefore the tied links are useful in the understanding  of classical links.
These invariants  for tied links  were  constructed by the Jones recipe\footnote{This terminology is the abstraction of the  method by which  V. Jones constructed the Homflypt polynomial, see \cite{joAM}.} and also by skein relations. In the construction using the Jones recipe,  the role played by the tied braid monoid is to tied links as the role of the braid groups   to the  classical links.
\smallbreak
With the aim  of  constructing  others classes of  tied  knot--like objects,   we  reformulate  the tied links in   algebraic--combinatoric terms, and    we  prove that the tied braid monoid has a certain decomposition as semi--direct product:  a part formed by ties (monoid of the set partitions) and the other part by the usual braid (braid group).  This decomposition and the  new   algebraic--combinatoric  context for tied links allows us  to introduce the {\it tied singular links} and {\it combinatoric tied singular links}.  Hence  we define four  families  of invariants for combinatoric  tied singular links which are constructed by the Jones recipe by using two maps from  the singular  braid monoid to the bt--algebra  and   two different presentations of this algebra.     These invariants are  five--variables polynomials of type Homflypt,    in the sense that they become  the Hompflypt polynomial whenever are evaluated on classical knots.
  We define  these  invariants  also  by  skein relations; the usual \lq local skein relations\rq , which take  into account any two crossing  strands, are replaced  by     \lq global  skein relations\rq , which   take  into account also the
  components  to which the crossing  strands  belong.

We also study here the behavior of these invariants, that is, we compare   them  with each other    and  with  another invariant for
singular links of type Homflypt polynomial, defined by Paris and Rabenda
in \cite{paraALIF} which is a four--variable polynomial that generalizes
the invariant defined by Kauffman y Vogel in \cite{kavoJKTR}.
As we said before, the importance of tied links lies in the fact that,
when evaluated on classical links, they are able to distinguish pairs
of isotopic
links not distinguished by classical polynomials, see   \cite{aijuMMJ1},
\cite{aijuMathZ},
\cite{aica} and \cite{chjukalaIMRN}. Now, we have to
notice that,  as far as we know, in literature there is not   a list of
non isotopic singular links which are not distinguished by the known
 invariants for  singular links. Therefore, we build  pairs of singular links starting
by some pairs of non isotopic classical links that are not distinguished by the
Homflypt polynomial, according to the list provided in \cite{chli},
then we calculate on them our invariants and the invariant due to
Paris and Rabenda \cite{kavoJKTR}.  We remark, finally,  that in general   it  seems to be not   easy to find pairs proving that the new polynomials are more
powerful on singular links.

\smallbreak
We give now   the layout of the paper. Section 2 establishes the main tools used during the paper, that is, some facts on set partitions and the bt--algebra. The main goal of Section 3, is to prove   Theorem \ref{decomposition}, which says that the tied braid monoid $TB_n$ can be decomposed as the semi--direct product,  denoted by  $\PP_n\rtimes B_n$,  between the monoid $\PP_n$, formed by  the set partition of  $\{1, \ldots ,n\}$   and the braid group $B_n$;  note that  the action of $B_n$ on $\PP_n$ is naturally  inherited  from the
 action of the symmetric group on $\PP_n$.  The decomposition of $TB_n$ as  semi--direct product uses several ideas of \cite{rhJAC} adapted to our situation. Now, the decomposition of $TB_n$ by the monoid $\PP_n$, a monoid eminently combinatoric, and the group $B_n$, induces to    treat  the tied links as    algebraic--combinatoric objects, {\it the combinatoric tied links}, which are introduced in Section 4;   we define also their  isotopy classes, which of course coincide  with those  of tied links. In Theorems  \ref{Alexander} and \ref{Markov} we prove, respectively,  the Markov and Alexander theorems for combinatoric tied links.
\smallbreak
In Section 5 we recall some elements from the the theory of singular links; also we introduce four families of invariants for singular links, see Theorems \ref{PsiPhi} and \ref{tildePsiPhi}. These are  five--variables polynomials, which we denote by
$\Phi_{\x,\y}$, $\Psi_{\x,\y}$, $\Phi_{\x,\y}'$ and $\Psi_{\x,\y}'$; notice that the letters $\x$ and $\y$ are two of the five variables of the invariants but they parametrize the invariants too. These invariants come out from the Jones recipe; more precisely, we construct homomorhpisms from the monoid of singular braids to the bt--algebra (Proposition \ref{homopsiphi}), so using these homomorphisms and the Markov trace on the bt--algebra \cite{aijuMMJ1}, we derive the invariants  after the usual method of rescaling and normalization originally due to V. Jones \cite{joAM}.
\smallbreak
Section 6 introduces the tied singular link  which is nothing more than a classical  singular link with ties, or,  equivalently, a  tied link with some singular crossings. We define then the combinatoric tied singular links, for short  cts--links. This definition (Definition \ref{cts}) is the natural extension of the combinatoric tied links (Definition \ref{comtiedlinks}). The algebraic counterpart of cts--links is provided: the monoid of tied singular links (Definition \ref{monotsb}). This monoid, denoted by $TSB_n$, is defined trough a presentation; however,   we prove  in Theorem \ref{wreathTSB}    that it can be obtained, in the same way as $TB_n$,  as a semidirect product,   denoted by  $\PP_n\rtimes SB_n$,  between   $\PP_n$ and  the singular braid monoid   $SB_n$ \cite{baLMP, biBAMS, smLN}. The section ends  proving, respectively, in Theorems \ref{AlexanderTSL} and \ref{MarkovTSL} the Alexander and Markov theorem for cts--links.
  \smallbreak
 Section 7 has two subsections: in the first one, we lifts the invariants $\Phi_{\x,\y}$, $\Psi_{\x,\y}$, $\Phi_{\x,\y}'$ and $\Psi_{\x,\y}'$ to cts--links, this is done simply by extending the domain of the defining morphism of these invariants from $SB_n$ to $TSB_n$. This is a simple matter since  $TSB_n$ is decomposed as $\PP_n\rtimes SB_n$, see Proposition \ref{thomopsiphi}. In the second subsection, we prove respectively in Theorems \ref{theoremPhi}, \ref{theoremPsi} and  \ref{theoremPhiPsiprime}, that the invariants $\Phi_{\x,\y}$, $\Psi_{\x,\y}$, $\Phi'_{\x,\y}$  and  $\Psi'_{\x,\y}$
can be defined through   skein rules.
\smallbreak
Section 8 is devoted to the  comparisons of the invariants defined here among them and and also with the four--variable polynomial invariant for singular links defined by Paris and Rabenda in  \cite{paraALIF}. Notably,  Theorem \ref{SS} clarifies  the differences between the $\Phi_{\x,\y}$'s the $\Psi_{\x,\y}$'s  with  respect to  the  parameters $\x$ and $\y$.  Finally,  in Theorems \ref{comparison1}, \ref{comparison2}  and  Propositions \ref{comparison3} and \ref{comparison4},  we  give examples showing that our invariants are   more powerful than   the Paris--Rabenda invariant.

\section{Preliminaries}
  In the present section we  recall principally  the  definitions and  main facts on set partitions and on  the bt--algebra. The paper is  in fact  based on these  objects.
\subsection{Set partitions}
  For $n\in {\Bbb N}$, we denote by ${\bf n}$ the set  $\{1, \ldots , n\}$ and by $\Psf$ the set formed by the set partitions of ${\bf n}$, that is, an  element of $\Psf$ is a collection    $I  = \{I_1, \ldots , I_k\}$ of pairwise--disjoint non--empty sets whose union is ${\bf n}$; the sets $I_1, \ldots , I_k$ are called the blocks of $I$; the cardinal of $\Psf$, denoted $b_n$, is called the   $n^{th}$ Bell number. Further,
$(\Psf , \preceq )$ is a poset with partial order defined as follows:
$I\preceq J$ if and only if each block of $J$ is  a union of blocks of $I$.

We shall use  the  following  scheme  of  a    set  partition in $\PP_n$,  according to the standard representation by arcs, see \cite[ Subsection 3.2.4.3]{MaBook}, that is:  the   point $i$ is   connected  by an  arc to  the  point $j$,    if  $j$ is   the minimum in the  same  block of $i$  satisfying $j>i$.   In   Figure \ref{Fig01}  a set partition in  representation by arcs.

\begin{center}
 \begin{figure}[H]
 \includegraphics{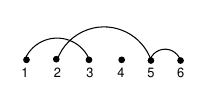}
 \caption{ Scheme of  the  partition  $I=\{\{1,3\},\{2,5,6\},\{4\}\}$. }\label{Fig01}
 \end{figure}
 \end{center}

The representation by arcs of a set partition   induces  a
natural indexation of its blocks. More precisely, we say  that  the blocks $I_j$'s of the set partition $I= \{I_1, \ldots , I_m\}$ of ${\bf n}$ are  {\it standard indexed} if $\min(I_j)< \min(I_{j+1})$, for  all $j$.   For instance, in the set partition of Figure \ref{Fig01} the blocks are indexed as:
$I_1= \{1,3\}$, $I_2=\{2,5,6\}$ and $I_3=\{4\}$.

As usual we denote by $S_n$ the symmetric group on $n$ symbols and we set  $s_i=(i, i+1)$.
The permutation action of $S_n$ on ${\bf n}$  inherits, in the obvious way,   an action of $S_n$ on $\Psf$  that is,
for $I=\{I_1, \ldots , I_m\}$ we have
\begin{equation}\label{SnOnPn}
w(I) :=\{w(I_1), \ldots , w(I_m)\}.
\end{equation}

Notice that this action   preserves  the cardinal of  each block   of  the  set partition.

We shall say that two set partitions $I$ and $I'$ in $\Psf$ are conjugate, denoted by $I \sim I'$, if there exits $w\in S_n$ such that, $I' = w(I) $; if it is necessary to precise   such   $w$, we write $I \sim_w I'$.
Further, observe that if  $I$ and  $I^{\prime}$ are  standard  indexed with  $m$  blocks,  then the  permutation $w$  induces a permutation    $w_{I,I'}$ 
of $S_m$  acting on   the indices of the blocks.

\begin{example}\label{exw} Let    $I=\{\{1,2\}_1,\{3\}_2,\{4,5\}_3,\{6\}_4\}$ and  $I^{\prime}=\{ \{1\}_1,\{2,5\}_2,\{3,6\}_3\{4\}_4\} $, so  $n=6$ and   $m=4$. We have
$I \sim_w I^{\prime}$, where:
$$
w=(1,6)(2,3,4,5)\quad \text{and}\quad   w_{I,I^{\prime}}=  (1,3,2,4).
$$
\end{example}

Given a  permutation  $w\in S_n$ and  writing $w=c_1\cdots c_m$ as
product of disjoint cycles, we denote by $K_w$  the set partition  whose blocks are the cycles $c_i$'s, regarded now as
subsets of $\bf n$.    Reciprocally, given a set partition $I=\{I_1, \ldots ,I_m\}$ of $\bf n$ we denote by $w_I$ an element of $S_n$ whose cycles are the blocks $I_i$'s.   Moreover, we shall say that the cycles of $w_I$ are standard indexed, if they are indexed according to the standard indexation of $I$.

\begin{notation} \rm
When there is no risk of  confusion,   we  will  omit in the  partitions the  blocks  with a  single  element.
 \end{notation}

\subsection{The bt--algebra}

Let $\u$ be  an indeterminate and set ${\Bbb K}= {\Bbb C}(\u)$.

\begin{definition} [See \cite{ajICTP1, rhJAC, aijuMMJ1}] The bt--algebra , denoted by $\E_n(\u)$, is defined by $\E_1(\u):={\Bbb K}$  and for $n\geq  2$ as   the unital associative ${\Bbb K}$--algebra, with unity $1$, defined by braid generators $T_1,\ldots , T_{n-1}$ and ties generators $E_1,\ldots , E_{n-1}$ subjected to the following relations:
\begin{eqnarray}
E_iE_j & =  &  E_j E_i \qquad \text{ for all $i,j$},
\label{bt1}\\
E_i^2 & = &  E_i\qquad \text{ for all $i$},
\label{bt2}\\
E_i T_j  & = &
 T_j E_i \qquad \text{for $\vert i  -  j\vert >1$},
\label{bt3}\\
E_i T_i & = &   T_i E_i  \qquad
\text{ for $\vert i - j\vert =1$},
\label{bt4}\\
E_iT_jT_i & = &  T_jT_iE_i \qquad \text{ for $\vert i  -  j\vert =1$},
\label{bt5}\\
E_iE_jT_i & = & E_j T_i E_j  \quad = \quad T_iE_iE_j \qquad \text{ for $\vert i  -  j\vert =1$},
\label{bt6}\\
 T_i T_j & = & T_j T_i\qquad \text{ for $\vert i - j\vert > 1$},
 \label{bt7}\\
 T_i T_j T_i& = & T_j T_iT_j\qquad \text{ for $\vert i - j\vert = 1$},
 \label{bt8}\\
 T_i^2 & = & 1  + (u-1)E_i + (u-1)E_i T_i\qquad \text{ for all $i$}.
 \label{bt9}
\end{eqnarray}
\end{definition}
Notice that every  $T_i$ is invertible,  and
\begin{equation}\label{Tinverse}
T_i^{-1} = T_i + (u^{-1} - 1)E_i + (u^{-1} - 1)E_iT_i.
\end{equation}

The bt--algebra is finite dimensional. Moreover, there is a basis defined by S. Ryom--Hansen;   we    describe here the construction of   this  basis,  because    some   elements  of it admit  analogous that  will be  used    in Section 2.

For $i<j$, we define $E_{i,j}$ by
 \begin{equation}\label{Eij}
 E_{i,j} = \left\{\begin{array}{ll}
 E_i & \text{for} \quad j = i +1\\
 T_i \cdots T_{j-2}E_{j-1}T_{j-2}^{-1}\cdots T_{i}^{-1}& \text{otherwise.}
 \end{array}\right.
 \end{equation}
 For any  nonempty  subset $J$  of $\bf n$ we define $E_J=1$  for  $|J|=1$ and otherwise by
$$
 E_J := \prod_{(i,j )\in J\times J, i<j}  E_{i,j}.
$$
Note that $E_{\{i,j\}} = E_{i,j}$.
For $I = \{I_1, \ldots , I_m\} \in \Psf$,  we define
 $E_I$  by
 \begin{equation}\label{EI}
 E_I = \prod_{k}E_{I_k}.
 \end{equation}
Now,  if   $w=s_{i_1}\cdots s_{i_k}$ is a reduced expression of $w\in S_n$, then the  element $T_w:= T_{i_1}\cdots T_{i_k}$ is well defined.
The action of $S_n$ on $\Psf$ is inherited from the  $E_{I}$'s and  we have:

\begin{equation}\label{w(E)}
T_w E_I T_w^{-1} =  E_{w(I)}\qquad (\text{see \cite[Corollary 1]{rhJAC}}).
\end{equation}

\begin{theorem}\label{tracebt}\cite[Corollary 3]{rhJAC} \label{basEn}
The set $\{E_I T_w  \,; \,w\in S_n,\, I\in  \Psf \}$ is a ${\Bbb K}$--linear basis of
${\mathcal E}_n(\u)$. Hence the dimension of ${\mathcal E}_n(\u)$ is $b_nn!$.
\end{theorem}

 The theorem above implies that  ${\mathcal E}_n(\u)\subseteq {\mathcal E}_{n+1}(\u)$, for all $n$. Denote ${\mathcal E}_{\infty}(\u)$ the inductive limit
associated to these inclusions and by $\rho$ the Markov trace defined on ${\mathcal E}_{\infty}(\u)$. More precisely,  fixing  two commutative independent variables  $\a$ and $\b$,  we have the following theorem.
\begin{theorem}\label{tracerho}\cite[Theorem 3]{aijuMMJ1}
There exists a  unique family $\rho :=\{\rho_n\}_{n\in {\Bbb N}}$, where  $\rho_n$'s are linear maps, defined inductively,  from $\mathcal{E}_n(\u)$ in   $\mathbb{K}[\a, \b]$   such  that  $\rho_n(1) =1$ and satisfying,  for all   $X, Y \in\mathcal{E}_n(\u)$,    the following rules:
\begin{enumerate}
\item $\rho_n (XY)=\rho_n(YX)$,
\item $\rho_{n+1}(XT_n)=\rho_{n+1}(XT_nE_n) = \a\rho_n(X)$,
\item $\rho_{n+1}(XE_n)= \b\rho_n(X)$.
\end{enumerate}
\end{theorem}

\begin{remark}\label{Vi}\rm
Extending the field ${\Bbb K}$ to ${\Bbb K}(\v)$ with $\v^2=\u$,
 we can define  (cf. \cite[Subsection 2.3]{maIMRN}):
 \begin{equation}\label{defVi} V_i := T_i + (\v^{-1}-1)E_iT_i.\end{equation} Then the $V_i$'s and the $E_i$'s  satisfy the relations (\ref{bt3})--(\ref{bt8}) and the quadratic relation (\ref{bt9}) is transformed in
\begin{equation}\label{newbt9}
V_i^2 = 1 + (\v - \v^{-1})E_iV_i.
\end{equation}
So,
\begin{equation}\label{Ttildeinverse}
V_i^{-1} = V_i - (\v - \v^{-1})E_i.
\end{equation}

In \cite{chjukalaIMRN, esryJPAA, jaPoJLMS}  this quadratic relation is used to define the bt--algebra. Although at algebraic level these algebras are the same, we will see  that they lead on  to different invariants. Thus, in order to distinguish these two presentations of the bt--algebra, we will write $\mathcal{E}_n(\v)$ when the bt--algebra is  defined by using the quadratic relation (\ref{newbt9}).
\end{remark}

\section{The tied braids monoid}

The goal of this section is to prove   Theorem \ref{TBnsemidirect} which says that the tied braid monoid  $TB_n$ \cite{aijuJKTR1}, defined originally by generators and relations,  can be realized as  a monoid constructed from   the monoid of set partitons of ${\bf n}$    and the braid group on $n$--strand.
\subsection{  The   monoid of set partitions }\label{monoidPn}

The set $\Psf$ has a structure of  commutative monoid  with  product $\ast$. More precisely,  the product $I\ast J$ between $I$ and $J$ is  defined as the minimal  set partition, containing $I$ and $J$,  according to  $\preceq$;  the identity of this monoid is $ {\bold 1}_n := \{\{1\}, \{2\}, \ldots ,\{n\}\} $. Observe that:
\begin{equation}\label{partial}
I \ast J = J, \quad \text{whenever}\quad  I\preceq J,
\end{equation}
\begin{equation}\label{astwithw}
I\ast J = I \ast w_I (J).
\end{equation}

For every $1\leq i < j\leq n$ with $i\not= j$, define $\mu_{i,j}\in \Psf$ as the set partition whose blocks are $\{i,j\}$ and $\{k\}$ where $1\leq k\leq n$ and $k\not=i,j$.    We shall write  $\mu_{i,j}\mu_{k,h}$  instead of $\mu_{i,j} \ast \mu_{k,h}$; we have
\begin{equation}\label{relationsPn}
\mu_{i,j}^2 = \mu_{i,j} \quad \text{and }\quad \mu_{i,j}\mu_{r,s} = \mu_{r,s}\mu_{i,j}.
\end{equation}
Moreover, we have the following proposition.
\begin{proposition}\label{presentationPn}
The monoid $\Psf$  is generated by set partitions  $\mu_{i,j}$'s.
 \end{proposition}

\begin{definition}\label{Pinf}
We denote by $\mathsf{P}_{\infty}$, the inductive limit monoid associated to the    family $\{(\Psf,\iota_n)\}_{n\in \mathbb{N}}$, where $\iota_n$ is the monoid monomorphisms from $\Psf$ into  $\mathsf{P}_{n+1}$,  such that for $I\in \Psf$,  the image $\iota_n(I)\in  \mathsf{P}_{n+1}$    is defined   by adding   to  $I$ the   block $\{n+1\}$. Observe that the inclusions   preserve  $\preceq$, that is, if $I\preceq J$ for  $I, J\in
\Psf$, then  $I\preceq J$  when  $I, J$ are considered as  elements of  $\mathsf{P}_{n+1}$.
\end{definition}

\subsection{The tied braid monoid}
  Denote   $B_n$ the braid group on $n$--strand, that is   the group presented by the elementary braids $\sigma_1 , \ldots , \sigma_{n-1}$ subjected to the following relations: $\sigma_i\sigma_j = \sigma_j\sigma_i$  for all $i,j$ s.t. $\vert i-j\vert >1$  and $\sigma_i\sigma_{i+1}\sigma_i = \sigma_{i+1}\sigma_i\sigma_{i+1}$ for $1\leq i\leq n-2$.

Recall now that we have a natural epimorphism $\pi: B_n \longrightarrow S_n$
defined by mapping $\sigma_i$ to $s_i$. We denote by $\pi_{\alpha}$ the image of $\alpha$ by $\pi$; thus
$\pi_{\sigma_i}= \pi_{\sigma_i^{-1}}=s_i$.
 The epimorphism  $\pi$  defines an action of $B_n$ on $\Psf$:  namely,  the result of  $\beta\in B_n$, acting on $I\in \Psf$, is $\pi_{\beta}(I)$, see (\ref{SnOnPn}).
This  action of $B_n$ on $\Psf$ defines   a monoid structure on  the cartesian product   $\Psf\times B_n $,   where the multiplication is defined as follows,
  \begin{equation} \label{mult}
(I, \alpha  ) (J, \beta ) = (I\ast \pi_\alpha(J), \alpha\beta ).
\end{equation}
We shall denote this monoid by $\Psf \rtimes B_n$.
Note that $B_n$ and $ \Psf$ can be regarded as  submonoids of     $\Psf\rtimes B_n$. More precisely, an  element $\beta	 \in B_n$ correspond to $(1, \beta)$ (which will be  denoted simply  by $\beta$ if  there  is  no  risk of confusion); an element $I\in \Psf $ corresponds to the element $\widetilde{I} := (I,1)$. The decomposition $(I,\beta) = (I,1)(1, \beta)$, together with the Proposition \ref{presentationPn}, implies that  $\Psf \rtimes B_n$ is generated by the $\widetilde{\mu}_{i,j}$'s and the  $\sigma_i$'s. Now, we also have,  by eq. (\ref{mult}):
\begin{equation}\label{sigmaI}
(1,\beta)(I,1)(1,\beta^{-1}) = (\pi_\beta(I), 1).
\end{equation}

  Thus, by taking $I=\mu_{i,i+1}$  and $\beta =\sigma_{j-1}\sigma_{j-1}\cdots \sigma_{i+1}$ with $j>i+1$, we deduce that every generator $\widetilde{\mu}_{i,j}$ can be written as a word in the $\widetilde{\mu}_{i, i+1}$ and $\sigma_{i+1}^{\pm 1},\ldots,  \sigma_{j-1}^{\pm 1}$,  since, for $j>i+1$
$$
\mu_{i,j}= s_{j-1} s_{j-2}\cdots s_{i+1} (\mu_{i,i+1}).
$$
 Hence we have the following lemma.
\begin{lemma}\label{generators}
The monoid $\Psf \rtimes B_n$ is generated by $\widetilde{\mu}_{1,2}, \ldots ,\widetilde{\mu}_{n-1, n}, \sigma_1^{\pm 1},\ldots ,\sigma_{n-1}^{\pm 1}$.
\end{lemma}

We will see below that   $ \PP_n \rtimes B_n$  is the tied braid monoid $T\!B_n$ introduced in \cite{aijuJKTR1}.

\begin{definition}\cite[Definition 3.1]{aijuJKTR1}
$T\!B_n$ is  the monoid generated by  the elementary  braids
 $\sigma_1^{\pm 1}, \ldots , \sigma_{n-1}^{\pm 1}$  and the generators   $\eta_1, \ldots ,\eta_{n-1}$, called ties,  such the $\sigma_i$'s satisfy braid relations among them  together with the following relations:
\begin{eqnarray}
\eta_i\eta_j & =  &  \eta_j \eta_i \qquad \text{ for all $i,j$},
\label{eta1}\\
\eta_i\sigma_i  & = &   \sigma_i \eta_i \qquad \text{ for all $i$},
\label{eta2}\\
\eta_i\sigma_j  & = &   \sigma_j \eta_i \qquad \text{for $\vert i  -  j\vert >1$},
\label{eta3}\\
\eta_i \sigma_j \sigma_i & = &   \sigma_j \sigma_i \eta_j \qquad
\text{ for $\vert i - j\vert =1$},
\label{eta4}\\
\eta_i\sigma_j\sigma_i^{-1} & = &  \sigma_j\sigma_i^{-1}\eta_j \qquad \text{ for $\vert i  -  j\vert =1$},
\label{eta5}\\
\eta_i\eta_j\sigma_i & = & \eta_j\sigma_i\eta_j  \quad = \quad\sigma_i\eta_i\eta_j \qquad \text{ for $\vert i  -  j\vert =1$},
\label{eta6}\\
 \eta_i\eta_i & = & \eta_i \qquad \text{ for all $i$}.
 \label{eta7}
\end{eqnarray}

\end{definition}

 Following the  construction of Ryom--Hansen's basis  we obtain that  the elements of $T\!B_n$ can be written in the form $\eta_I\beta$, where  $\beta\in B_n$ and  $\eta_I$'s are defined  analogously  to the $E_I$'s. We are going now to explain this fact.

  As in (\ref{Eij}),  for $1\leq i<j\leq n$, we put  $\eta_{i,j} := \eta_i$ for $j=i+1$, and otherwise
$$
\eta_{i,j} := \sigma_i\sigma_{i+1}\cdots \sigma_{j-2}\eta_{j-1}\sigma_{j-2}^{-1} \cdots \sigma_{i-1}^{-1}\sigma_i^{-1}.
$$
For every $i,j$ define  $\eta_{\{i,j\}}= \eta_{\min\{i,j\},\max\{i,j\}}$. We get  the  following lemma.
\begin{lemma}\label{lemmaEtas}
For all $i<j$ and $k$ we have:
\begin{enumerate}
\item $ \sigma_k\eta_{i,j}\sigma_k^{-1} = \eta_{\{s_k(i),s_k(j)\}}$,
\item $ \sigma_k^{-1}\eta_{i,j}\sigma_k= \eta_{\{s_k(i),s_k(j)\}} $,
\item $   \eta_{i,j} = \sigma_{j-1}^{-1} \cdots \sigma_{i+1}^{-1}  \eta_{i}   \sigma_{i+1}\cdots \sigma_{j-1}$,
\item  $\alpha\eta_{i,j}= \eta_{\{s(i), s(k)\}}\alpha$,   for $\alpha\in B_n$ and $s:=\pi_{\alpha}$,
\item The elements $\eta_{i,j}$'s are commuting  and  idempotent,
\item   $\eta_{i,j}\eta_{j,k}=\eta_{i,j}\eta_{i,k}=\eta_{i,k}\eta_{j,k}=\eta_{i,j}\eta_{j,k}\eta_{i,k}$ for all $i<j<k$.
\end{enumerate}
\end{lemma}
\begin{proof}
The proof of claims (1) and (2) are the same as the proof of \cite[Lemma 2]{rhJAC} but using now relation (\ref{eta5}) instead \cite[Lemma 1]{rhJAC}.

 Claims (3) and (4) are  direct consequences  of (1) and (2).

The proof of (5) is   analogous  to the proof of \cite[Lemma 3]{rhJAC}.

  The proof of (6)  is  contained in the  proof  of  \cite[Lemma 5]{rhJAC}.
\end{proof}

For every (non--empty) subset $M$  of $\bf n$, we define $\eta_M$= 1 if $\vert M\vert =1$, otherwise
\begin{equation}\label{etaM}
\eta_M := 	\prod_{(i,j)\in M^2 : i<j}\eta_{i,j}.
\end{equation}
Now, for $I=\{I_1, \ldots , I_m\}\in \mathsf{P}_n$, define $\eta_I$ as follows
\begin{equation}\label{etaI}
  \eta_I := \prod_j \eta_{I_j}.
\end{equation}

Let $X \subseteq {\bf n}\times {\bf n}$.  Observe  that  $X$  defines   an  equivalence relation on   ${\bf n}$ by setting: $i\sim_X i$
  and $i\sim_X j$  if and only if   there  is  a  chain  $i=i_1,i_2, \dots, i_m=j$ with $m>1$ in $X$ such that  either  $(i_r, i_{r+1})\in X $ or  $(i_{r+1},i_r)\in X$.
 Denote $\langle X\rangle$ the partition  of ${\bf n}$ determined by $\sim_X$.
\begin{lemma}\label{<X>}
For $X \subseteq {\bf n}\times {\bf n}$, we have
\begin{equation}\label{etaX}
 \eta_{\langle X\rangle} = \prod_{(i,j)\in X}\eta_{\{i,j\}}.
\end{equation}

\end{lemma}
\begin{proof} It follows  from  claim (6) of  Lemma \ref{lemmaEtas},  see   \cite[Lemma 5]{rhJAC}.
\end{proof}

\begin{proposition}\label{decomposition}
The elements of $T\!B_n$ can be written in the form $\eta_I\beta$, where $I\in \mathsf{P}_n$ and $\beta\in B_n$.
\end{proposition}
\begin{proof}

Every element $\nu$ in $T\!B_n$ is a word of the form   $\nu_{1}\cdots \nu_{m}$, where each $\nu_{i}$ is equal to some $\eta_{k}$ or  some $\sigma_{k}^{\pm 1}$, with $k<n$.  Now, from (4) of Lemma \ref{lemmaEtas} it follows that every $\eta_{k}$ can be moved to the beginning of  the word, resulting then that $\nu$  has  the form $\eta\beta$, where $\eta$ is a product of $\eta_{i,j}$'s and $\beta\in B_n$. After, define $X$ as the set $\{(i,j): \eta_{i,j}\text{ appears in}\, \eta\}$. Then, Lemma  \ref{<X>} implies that  $\langle X\rangle$ is the set partition such that $\eta =\eta_{\langle X\rangle}$.

\end{proof}
\begin{theorem}\label{TBnsemidirect}
The tied braid monoid  $T\!B_n$ is the monoid  $ \Psf\rtimes B_n$.
\end{theorem}
\begin{proof}
The mapping $\sigma_i\mapsto \sigma_i$, $\eta_i \mapsto \widetilde{\mu}_{i,i+1}$ defines a morphism $\phi$  of monoids   from $T\!B_n$ to  $ \Psf\rtimes B_n $, since   $\phi$ respects the defining relations of $T\!B_n$; e.g.,  we shall  check relation (\ref{eta5}):
$$
\phi(\sigma_j)\phi(\sigma_i^{-1})   \phi(\eta_j) =  (1,\sigma_j)(1, \sigma_i^{-1})(\mu_{j,j+1}, 1)
= (1,\sigma_j\sigma_i^{-1})(\mu_{j,j+1}, 1)
= ( s_js_i(\mu_{j,j+1}),\sigma_j\sigma_i^{-1}).
$$
Now, for $\vert i-j\vert =1$,  $s_j s_i(\mu_{j,j+1}) = \mu_{i,i+1}$ ; then
$$
 \phi(\sigma_j)\phi(\sigma_i^{-1}) \phi(\eta_i) =   (\mu_{i,i+1},\sigma_j\sigma_i^{-1}) =
 (\mu_{i,i+1},1)(1, \sigma_j)(1,\sigma_i^{-1})= \phi(\eta_i)\phi(\sigma_j)\phi(\sigma_i^{-1}).
$$
Thus, from Lemma \ref{generators} we  get  that $\phi$ is an epimorphism. The proof of the proposition will be  completed by proving that  $\phi$ is a monomorphism, which is  done as follows.

 Let $a$ and $b$ in $T\!B_n$ such that $\phi(a)=\phi(b)$. According to Proposition \ref{decomposition}, we can write: $a= e_I\alpha$ and $b=e_J\beta$, where $I, J\in \mathsf{P}_n$ and $\alpha, \beta\in B_n$. Then $\phi(a)=\phi(b)$ is equivalent to
$\phi(e_I)(1, \alpha) = \phi(e_J)(1,\beta)$; now, since $\phi(e_I)$ and $\phi(e_J)$ are words in the $\widetilde{\mu}_{i,j}$'s, it follows that $\alpha =\beta$; thus, it remains only to prove $e_I=e_J$.
 To do this,  note  that $\phi(\eta_{i,j}) = \widetilde{\mu}_{i,j}$; then, we deduce that for any subset  $M$   of $\bf n$:  $\phi(\eta_M) = \widetilde{M}$. Hence $\phi(e_I) = \widetilde{I}$, for all $I\in \mathsf{P}_n$. Therefore,
 $\phi(e_I)(1, \alpha)= \phi(e_J)(1,\alpha)$,   so  that   $ a = b $.
\end{proof}

\begin{remark}\rm \label{TBinf}
The natural inclusions $B_n\subseteq B_{n+1}$ together with the inclusions $\mathsf{P}_n\subseteq \mathsf{P}_{n+1}$ (see  Definition \ref{Pinf}) induce  the tower of monoids $T\!B_1\subseteq T\!B_{2}\subseteq \cdots \subseteq T\!B_{n}\cdots $. We will denote by  $TB_\infty$ the inductive limit associated to this tower. Notice that $\mathsf{P}_{\infty}$ and $B_{\infty}$ can be regarded as submonoid of $TB_\infty$.
\end{remark}

\subsection{Diagrams}   As for the braid group, we can use diagrams to represent the the elements of the tied braid monoid. This diagrammatic  representation is used later in the paper and works under the  conventions listed below.

\begin{enumerate}
\item The  multiplication in $B_n$ is done  by  concatenation, more  precisely, the product  $\beta_1 \beta_2$ is done  by putting  the  braid $\beta_1$ over the braid  $\beta_2$,  so  that   a  word in  the  generators  has to  be read  from top to  bottom.
\item The  tied braid $(I,\alpha)$ is  represented as  the  braid  $\alpha \in B_n$ with the partition $I$ of the strands at  the  top of  $\alpha$,  see Figure \ref{Fig02}.
\item The permutation $\pi_{\beta}$,   defined  by  the  braid $\beta$,  acts  on the set of  $n$ strands at  the  bottom of $\beta$.
\end{enumerate}

\begin{center}
 \begin{figure}[H]
 \includegraphics{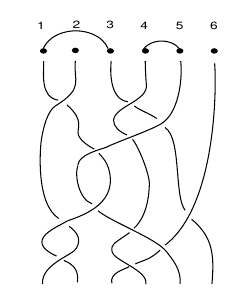}
 \caption{ Diagrammatic  representation of  an element of  $\Psf \rtimes B_n$.}\label{Fig02}
 \end{figure}
 \end{center}

\section{Tied links  and combinatoric  tied  links }

We start this section recalling  briefly the tied links, later we introduce their combinatoric version, called   combinatoric tied links. Then we reprove the Alexander and Markov theorems    for them.

\subsection{Tied links}
 Tied links were introduced in \cite{aijuJKTR1} and roughly correspond   to  links  with  ties connecting pairs of  points  of  two  components or of the  same  component.
The   ties   in the picture of the tied links  are   drawn  as  springs,  to outline (diagrammatically)  the fact  that they  can be contracted  and  extended,  letting  their  extremes  to  slide  along the  components.

  We will use the notation $C_i \leftrightsquigarrow C_j$  to indicate that either there is a tie between the components $C_i$ and $C_j$ of a link, or
  $C_i$ and $C_j$ are  the     extremes of  a  chain  of $m>2$ components $C_1, \dots , C_m$, such that there is  a  tie  between $C_i$  and  $C_{i+1}$,  for  $i=1,\dots, m-1$.

  \begin{definition}\cite[Definition 1.1]{aijuJKTR1}\label{TL} \rm   Every 1--link is by definition a tied 1--link.   For $k > 1$, a  tied $k$--link  is a   link  whose  set of $k$ components is partitioned  into parts according to: two components $C_i$ and $C_j$ belong to  the  same   part  if
  $C_i \leftrightsquigarrow C_{i+1}$.
\end{definition}
 Therefore, a tied  $k$--link $L$,   with  components' set  $C_L=\{C_1, \ldots , C_k\}$, determines a pair $(L, I(C_L))$ in $\mathcal{L}_k \times \Psk$, where $i$ and $j$ belong to the same block of $I(C_L)\in \Psk $  if   $C_i \leftrightsquigarrow C_j$.  In   Figure \ref{Fig1}  two  tied links with four components  are shown with the  corresponding partitions.

A tie of a tied link   is said {\sl essential}  if cannot  be  removed without modifying  the  partition $I(C_L)$, otherwise  the  tie is  said  {\sl  unessential}, cf. \cite[Definition 1.6]{aijuJKTR1}.  Observe    that  between  the $c$ components indexed by   the  same block  of  the set  partition,  the  number of       essential ties    is  $c-1$; for instance, in the tied link of  Fig. \ref{Fig1}, left, among the  three  ties  connecting  the first three components, only two  are  essential. The  number of  unessential ties  is  arbitrary.   Ties  connecting     one  component  with itself are  unessential.


\begin{center}
 \begin{figure}[H]
 \includegraphics{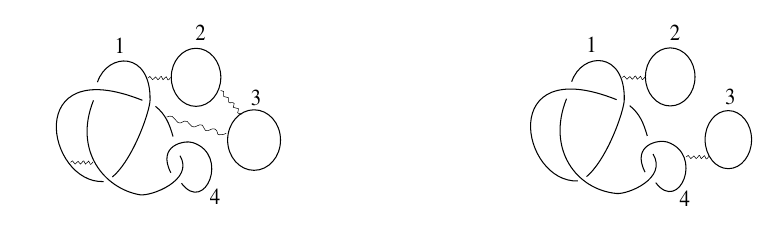}
 \caption{ Left: $I_1=\{\{1,2,3\},\{4\}\}$  and right:  $I_2=\{\{1,2\},\{3,4\}\}$.  }\label{Fig1}
 \end{figure}
 \end{center}

\subsection{Combinatoric tied links }
A combinatoric  tied  link  is   a link provided  with a partition of its set of  components.
We  will  depict  a  combinatoric  tied  link  as  a  link  with  numbered  components  and  the  scheme of  a   partition (see  Figure \ref{Fig1a}).  We  define now the concept of {\it t--isotopy} of combinatoric tied links    which reflects   the t--isotopy of tied links.

 Let $\mathcal{L}$ be the set  formed by  the links in $\mathbb{R}^3$. We shall denote $\mathcal{L}_k$ the set of links with  $k$ components. Hence,   $\mathcal{L}= \coprod_{k\in \mathbb{N}}\mathcal{L}_k$.

Observe  that the numbering of  the components of a link is arbitrary.  Now, an isotopy  between two  links $L$ and  $L'$ in   $\mathcal{L}_k$,  defines a  bijection from  the  set of components of  the first  to  the  set of  components of  the  second; we denote   such bijection by $w_{L,L'}$.

\begin{definition}\label{comtiedlinks}
  An element  of $ \mathcal{L}_k^t :=\mathcal{L}_k \times \Psk$ is called  $k$--tied combinatoric link; then,  combinatoric tied links  are the elements of $\mathcal{L}^t$, where
$$
\mathcal{L}^t :=\coprod_{k\in \mathbb{N}}\mathcal{L}_k^t.
$$
  In  what follows,   we    denote   by  $(L,I(C_L))$  the combinatoric   tied  link in which the link $L$ has components set
  $C_L$  with   set  partition  $I(C_L)$.
\end{definition}

  Note  that  a  classical  link $L \in \mathcal{L}_k$ with  components  set $C_L$ can be  considered as  a combinatoric   tied  link   $(L,  {\bold 1}_k )$.

  \begin{definition}  Two  partitions    $I(C_L)$ and  $I(C_{L^{\prime}})$  of  two isotopic  links $L$ and  $L'$  are  said iso--conjugate whenever $I(C_L)\sim_{ w_{L, L'}} I(C_{L^{\prime}})$.
\end{definition}

  \begin{definition}
  We will say that two tied links $(L, I(C_L))$ and $(L^{\prime}, I(C_{L^{\prime}}))$ are t--isotopic  if $L$ and $L^{\prime}$ are ambient isotopic and $I(C_L)$ and $I(C_{L^{\prime}})$   are  iso--conjugate.
\end{definition}

\begin{proposition} The   t--isotopy   relation, denoted by $\sim_t$, is an equivalence relation on $\mathcal{L}^t$. \end{proposition}

In the sequel  we do not distinguish formally between a tied link  and its class of t--isotopy.

The equivalence between the concepts of  tied links and combinatoric tied links is clear, e.g.  compare Figures  \ref{Fig1}  and \ref{Fig1a}.

\begin{example}\label{ex1} \rm See  Figure    \ref{Fig1a}.

\begin{center}
 \begin{figure}[H]
 \includegraphics{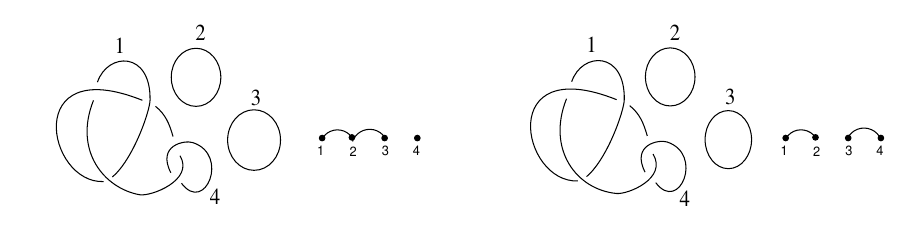}
 \caption{The combinatoric  tied links $(L,I_1 )$  and  $(L,I_2)$ corresponding to the  tied links of  Figure \ref{Fig1}. }\label{Fig1a}
 \end{figure}
 \end{center}

\end{example}

 \subsection{  Alexander and Markov  theorems for  combinatoric  tied  links}

The    algebraic  counterpart of   tied links  is the tied braid monoid $T\!B_\infty$ introduced in \cite{aijuJKTR1}    More precisely, in this paper     we  have proven the   Alexander  and Markov theorems   for  tied  links.
  Below we  reprove these theorems  but   regarding   the tied braid monoid   $T\!B_n$  as \lq the semi--direct  product\rq\  $\Psf \rtimes B_n$  and the tied links as combinatoric tied links.

 \begin{definition} \label{closure} The closure of the tied braid $(I, \alpha)$, denoted by $\widehat{(I, \alpha)}$, is   the combinatoric tied link $(L,J)$,  where $L=\widehat{\alpha}$ is the usual closure of the braid $\alpha$, done, as usual, by identifying
the bottom with the top of the strands of $\alpha$,  whereas  the  partition $J$  is defined by  the  partition  $I$ and the permutation $\pi_\alpha$,  as  explained  below.

\begin{center}
 \begin{figure}[H]
 \includegraphics{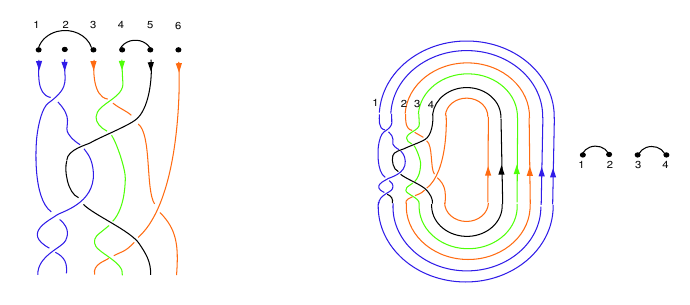}
 \caption{The  closure of  the  tied  braid $(I,\alpha)$ at left is  the  combinatoric tied link  $(L,J)$ at  right,  where the four components are distinguished by  different colors.  }\label{Fig10}
 \end{figure}
 \end{center}
More precisely, if $k$ denotes  the number of components of the link
 $\widehat{\alpha}$, or equivalently  the  number of cycles of  the permutation $\pi_\alpha$, then
$J$  is the set partition of $\bf k$ whose blocks are  determined by those arcs of $I$ connecting strands belonging to different cycles of $\pi_\alpha$.
  For  instance, in Figure \ref{Fig10} the  arc (1,3) of $I$ connecting   the blue    and the red  components,  determines  the arc   (1,2) of $J$.
\end{definition}

 The extension of the  Alexander and Markov  theorems to   combinatoric   tied links, i.e. the characterization of the class of tied braids whose closures give the same   combinatoric tied link,
must take into  account the  behavior of  the partition $I$  under  closure of  the  tied braid  $(I,\alpha)$.
For  this  reason,  before  of stating    the   Alexander and Markov theorems for combinatoric tied links,  we need to introduce   the  tools   below.

\begin{definition}
Let  $I$, $K\in \Psf$  such that  $ K\preceq I $, let  $m\le n$  be  the  number of  blocks of  $K$ and  ${\bf m}=\{1, \ldots , m\}$. We denote by    $I/K$  the set partition of   ${\bf m}$,
 whose      blocks    are the sets
 $$
   (I/K)_i:= \{j \in {\bf m} \, : \, K_j \subseteq I_i  \} ,
 $$
  where the blocks $K_j$'s and $I_i$'s are taken standard indexed.
\end{definition}

\begin{example}  Let  $I=\{\{1,2,5\},\{3,4\}\}$,  $K=\{\{1,2\},\{3,4\},\{5\}\}$.  Then  $m=3$, $K_1=\{1,2\}, \  K_2= \{3,4\},\ K_3=\{5\}$   and $I/K=\{\{1,3\},\{2\}\}$.
\end{example}

\begin{proposition}
For $I\in \Psf$ with $k$ blocks, we have:
\begin{enumerate}
\item $I/I ={\bf 1}_k$,
\item $I/{\bf 1}_n = I$.
\end{enumerate}
\end{proposition}

\begin{definition}
Let  $K\in \Psf$  with $m$ blocks   standard indexed  $K_1,\ldots, K_m$  and  $J\in \PP_m$ with  $l$ blocks $J_1,\dots   , J_l$  standard indexed, too.
We   denote by $ K  \times J $ the  set partition in $\PP_n$    with   $l$ blocks  $(K  \times J)_i$'s  given by
$$
  (K  \times J)_i =\cup_{j\in J_i} K_j.
$$
Notice that  $K\preceq   K\times J $.
 \end{definition}

\begin{example}
 Let    $K=\{\{1,2\},\{3,4\},\{5\}\}$,   $m=3$,  and  $J:= \{\{1,3\},\{2\}\}$.  Then  $K \times J=\{\{1,2,5\}, \{3,4\}\}$.
\end{example}

\begin{notation}\label{Kalpha} \rm
Given    a  braid  $\alpha\in B_n$, we  denote by  $K_\alpha\in\Psf$  the set partition    whose  blocks are    the cycles of  the  permutation $\pi_\alpha$,   including  the 1--cycles. \end{notation}

\begin{remark}\label{cycles} \rm
  Recall that the closure  of  a  classical  braid $\alpha$ is  a link whose  components  are in one--to--one  correspondence  with  the  cycles of  the permutation  $\pi_\alpha$.  The {\it standard indexation} of the components  of $\widehat{\alpha}$ is   that obtained  from the standard  indexation of the cycles of $\pi_{\alpha}$.
 \end{remark}

\begin{example} Consider the  braid  $\alpha$ in Figure \ref{Fig10}, left.  We  have $\pi_\alpha= (1,2)(3,6)$, so  $K_\alpha=\{\{1,2\},\{3,6\},\{4\},\{5\}\}$. The  four  blocks correspond to  the  components of the  link at  right.
\end{example}
In order  to  distinguish  a  set  partition  $I\in\Psf$, associated  to  a  tied braid  $(I,\alpha)$,
from  a   set partition  $J\in \Psk$,  associated  to  a  tied  link  $(L,J)$,  we  shall
call  this  last partition  sc--partition (from  set of components).

  For   $(I,\alpha)\in T\! B_n$    we define
 \begin{equation}\label{barI}
  \overline{ I}_\alpha  := I \ast K_\alpha .
  \end{equation}

\begin{proposition}\label{scpartition}
 If the    $k$--tied link $(L,J)$ is  the  closure  of  the tied  braid $(I,\alpha)$, then   the sc--partition  $J$  is given by
\begin{equation}\label{J}
   J =  \overline{ I}_\alpha  / K_\alpha .
   \end{equation}
\end{proposition}

\begin{proof}
The  number of   blocks  of $K_\alpha$, coincides with the  number of components of $L$, i.e., $k$.
  If   $\overline I$  has   $m$ blocks, we have   $m\le k$; moreover,  since $\overline I \preceq  K_\alpha$,  every block of $K_\alpha$ is contained  in a  block  of $\overline I$.
Therefore, $\overline{ I}_\alpha/K_{\alpha}$ is a set partition of  $\bf k$  having  $m$ blocks.   Now,
by definition, the block $i$ of  this set partition is
$$
 (\overline{ I}_\alpha/K_{\alpha})_i=
\{j\in {\bf k}\, ;\, (K_{\alpha})_j\subseteq {\overline I}_i\}.
$$
In other words, the  elements of the  set $(\overline{ I}_\alpha/K_{\alpha})_i$ are  the different blocks  of $K_\alpha$,
contained in the    block $(\overline{I}_\alpha)_i $.  Therefore,  an  arc of the   set partition  $\overline I$, connecting two elements of $\bf n$ belonging  to  a  same  block of $K_\alpha$,   does not determine an arc   in $\overline{ I}_\alpha/K_{\alpha}$. On the other hand,   any arc  of $\overline I$  connecting elements belonging to two  different blocks of $K_\alpha$,   determines  an arc   of    $\overline{ I}_\alpha/K_{\alpha}$.  Therefore  we conclude that  $\overline{ I}_\alpha/K_{\alpha}=J$.
\end{proof}

 \begin{example}
 \rm
 Fig. \ref{Fig2} shows at left  a  tied braid $(I,\alpha)$,  where  $I=\{\{1,3\},\{2\},\{4,5\},\{6\}\}$;    in the middle  the  tied braid $ (K_\alpha,\alpha) $, where  $K_\alpha= \{K_1=\{1,2\},K_2=\{3,6\},K_3=\{4\},K_4=\{5\}\}$,   so  that  $k=4$;  at  right, the  tied braid $(\overline I,\alpha)$, where  $\overline I=\{\{1,2,3,6\},\{4,5\}\}$. Observe that   $\overline I$ is made by 2 blocks,  the  first   containing the blocks $K_1, K_2$ and  the  second   containing the blocks  $K_3$ and $K_4$ of  $K_\alpha$.
We thus have that the sc--partition $J\in  \mathsf{P}_4$ is given by $\{\{1,2\},\{3, 4\}\}$. Observe  that  the closure of  $(I,\alpha)$ is  the combinatoric tied  link  $(L, I_2)$ shown in Fig. \ref{Fig1a}, right.

\begin{center}
 \begin{figure}[H]
 \includegraphics{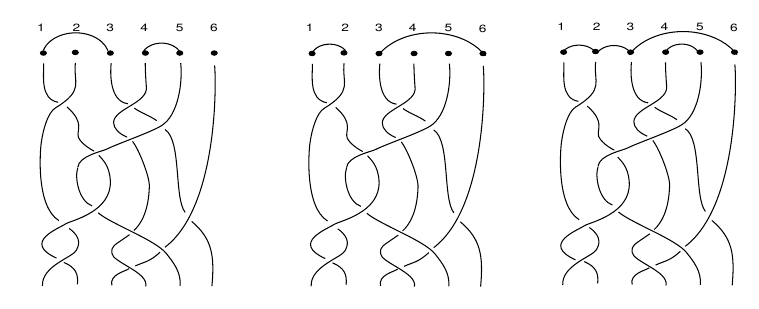}
 \caption{ }\label{Fig2}
 \end{figure}
 \end{center}

\end{example}

  We are ready now to prove the Alexander and Markov theorems in the context of combinatoric tied links.

\begin{theorem}
[Cf. {\cite[Theorem 3.5]{aijuJKTR1}}]\label{Alexander}
Every combinatoric  tied  link  can  be  obtained  as  closure of     a tied braid.  More precisely, if the link $L$ is  the  closure of the braid  $\alpha$,  then  the  combinatoric  tied link $(L,J)$,   up to a renumbering of  the  components,  is  the  closure  of the  tied braid  $(I,\alpha)$,
where
\begin{equation}\label{IfromKJ}
I := K_\alpha   \times   J .
\end{equation}
\end{theorem}

\begin{proof}  Let  $(L,J)$ be a  combinatoric  tied  link. Applying   the Alexander  theorem  to  the   link $L$  we  get  a  braid  $\alpha$ whose closure  is $L$.   The standard indexed set partition $K_\alpha$ (see Remark  \ref{cycles}) defines an ordering of the $k$ components of  the closure of $\alpha$.
On  the other hand,  the set partition $J$  is  defined on the set of components ordered arbitrarily. By numbering the components  of $L$,  according to  the standard ordering of the blocks of  $K_\alpha$, we obtain from $J$ the  partition $\tilde J$.  Then  the set partition $I$ of  the  tied  braid $(I,\alpha)$ is  obtained  as    $K_\alpha  \times  \tilde J$.
\end{proof}

\begin{lemma}\label{lemmaM3}
Let  $(I,\alpha)\in T\! B_n$.
We have
\begin{equation}\label{forMarkovStab}
\overline{ I}_{\alpha} /K_{\alpha}= \overline{ I}_{\alpha\sigma_n^{\pm 1}}/K_{\alpha\sigma_n^{\pm 1}}.
 \end{equation}

\end{lemma}
\begin{proof}
  Firstly note that the block of $K_{\alpha\sigma_{n}^{\pm 1}} $ containing $n$ also contains $n+1$; thus the set partitions $K_{\alpha}$ and $K_{\alpha\sigma_{n}^{\pm 1}} $ differ only in the bock that contains $n$. Secondly, we deduce
then  that $\overline{I}_{\alpha}$ and $\overline{I}_{\alpha\sigma_{n}^{\pm 1}}$ also differs only in the block that contains $n$. Thus, equation (\ref{forMarkovStab})  follows.
\end{proof}

 \begin{theorem}[{Cf. \cite[Theorem 3.7]{aijuJKTR1}}]\label{Markov}
Denote  by $\sim_{tM}$ the equivalence relation on   $T\!B_{\infty}$ generated by the  following  replacements (or moves):

\begin{enumerate}
\item[M1.] t--Stabilization: for all $(I,\alpha)\in TB_n$, we can do the following replacements:
$$
(I,\alpha ) \quad \text{replaced by}\quad   (I,\alpha)(\mu_{i,j},1) \quad \text{if} \quad    i,j \quad  \text{belong to the  same  cycle of} \quad  \pi_\alpha ,
$$
\item[M2.] Commuting in $TB_n$: for all $(I_1,\alpha), (I_2,\beta)\in TB_n$, we can do the following replacement:
$$
(I_1,\alpha)(I_2,\beta )\quad\text{replaced by}\quad (I_2,\beta ) (I_1,\alpha)  ,
$$
\item[M3.] Stabilizations: for all $(I,\alpha) \in TB_n$, we can do the following replacements:
$$
 (I,\alpha) \quad \text{replaced by}\quad(I,\alpha \sigma_n ) \quad \text{or} \quad(I,\alpha \sigma_n^{-1}  ) .
$$
\end{enumerate}

  Then,   $(I, \alpha) \sim_{tM} (I,\beta)$ if and only if $ \widehat{(I,\alpha)} \sim_t \widehat{(I',\beta)}$.
\end{theorem}


\begin{proof}
  Firstly,  we prove that
    the closure of a tied braid does not change under the replacement of M1, M2 and M3.   Consider    the replacement M1 on $(I,\alpha)$:  according to Proposition \ref{scpartition}, the set partition corresponding to the combinatoric tied link $\widehat{(I,\alpha)(\mu_{ij}, 1)}$ is given by
$$
((I\ast \pi_{\alpha}(\mu_{i,j})) \ast K_{\alpha})/K_{\alpha}.
$$
But $(I\ast  \pi_{\alpha}(\mu_{i,j})) \ast K_{\alpha} = I\ast  K_{\alpha} $, since $  \pi_{\alpha} (\mu_{i,j})\preceq K_{\alpha}$, see (\ref{partial}).  Thus, the closures of $(I, \alpha)$ and $(I,\alpha)(\mu_{ij}, 1)$ have the same sc--partition.

Secondly, we   check that
$(\widehat{\alpha\beta}, J_1):=\widehat{(I_1,\alpha)(I_2,\beta)}$ and   $(\widehat{\beta\alpha}, J_2):=(I_2,\beta)(I_1,\alpha)$ are t--isotopic. Indeed, by Proposition \ref{scpartition}:
$$
J_1 =  ((I_1\ast \pi_{\alpha}(I_2)) \ast K_{\alpha\beta})/K_{\alpha\beta}\quad \text{and}\quad
J_2 = ((I_2\ast \pi_{\beta}(I_1))\ast K_{\beta\alpha})/K_{\beta\alpha}.
$$
Applying $\pi_{\beta}$ to   the  right member of the  first equality,  we get
$$
  (\pi_{\beta}(I_1)\ast \pi_{\beta}(\pi_{\alpha}(I_2))  \ast \pi_{\beta}(K_{\alpha\beta}))/\pi_{\beta}(K_{\alpha\beta}).
$$
Notice now that $\pi_{\beta}(K_{\alpha\beta}) = K_{\beta(\alpha\beta)\beta^{-1}} = K_{\beta\alpha}$. Then,
applying now  (\ref{astwithw}) to $\pi_{\beta\alpha}(I_2) \ast K_{\beta\alpha}$ in the last expression, we obtain
$$
 \pi_{\beta}(I_1)\ast
(I_2 \ast K_{\beta\alpha})/K_{\beta\alpha}  =  J_2 .
$$
Hence,   setting  $K:= K_{\alpha\beta}$ and $K':=K_{\beta\alpha}$,  we  have $J_2= w_{K,K'}(J_1)$, so  that  the  sc--partitions  $J_1$  and  $J_2$ are iso--conjugate; this, together with the fact that  $\widehat{\alpha\beta}$ and $\widehat{\beta\alpha}$ are isotopic,   implies  that  $\widehat{(I_1,\alpha)(I_1,\beta)}$ and   $\widehat{(I_2,\beta)(I_1,\alpha)}$ are t--isotopic.

Finally, notice now that Lemma \ref{lemmaM3}  shows that the replacement M3 on $(I,\alpha)$ does not affect its closure.

To prove the   statement  in the other direction,
   let us suppose  that    two $t$--isotopic  combinatoric tied links $(L,J)$  and $(L',J')$    are  the  closures of  two  tied  braids $(I,\alpha)$  and $(I',\alpha')$.  We have to prove    that $(I,\alpha)\sim_{tM}(I',\alpha')$.   We  suppose  that  the ordering  of the  components in $J$  and  $J'$  corresponds, respectively,  to  that  induced  by  $K_\alpha$  and  $K_{\alpha'}$.

 Now,  from  the  Markov  theorem for  classical  links  we  know   that  the  braids  $\alpha$  and  $\alpha'$ are    Markov equivalent, i.e.,   they are  related  by  a  sequence of    replacements M2   and/or  M3, where    the set partitions are  neglected.  From (\ref{IfromKJ}), we  have
$$
I=K_\alpha \times  J    \quad \text{and}\quad  I'= K_{\alpha'} \times J'.
$$
  Observe   also  that  $J$ and  $J'$  are  set partitions iso--conjugate  of $\bf{k}$,  $k$ being the  number of  components of  $L$ and $L'$; we  write $J'=w(J)$, with $w\in S_k$.  On the other hand,  $K_\alpha$ and  $K_{\alpha}'$  are  set partitions    with $k$ blocks, respectively, of some  $\bf{n}$  and  $\bf{n'}$. Since the M1 replacement  does not affect  the  partition   $K_\alpha$,
we have      to prove  that   the  sequence of  replacements    M2  and/or   M3  that transform  $\alpha$ into  $\alpha'$,    transforming by consequence   the partition  $K_\alpha$ into $K_{\alpha'}$, induce the  permutation $w$   above.  Indeed, observe  firstly that    the  set partition $K_\alpha$   is transformed step by  step  into a  sequence  of  $  r$ set partitions $K_{\alpha_j}$  (with   $K_{\alpha_1}=K_{\alpha}$ and  $K_{\alpha_r }=K_{\alpha'}$)   as  long as   $\alpha$  is  transformed by moves  M2 and/or M3 in  the  sequence  $\alpha_j$, with   $\alpha_1 =\alpha$ and  $\alpha_r =\alpha'$.   Secondly, notice  that each partition  $K_{\alpha_j}$  has $k$  blocks,  and  that for every pair $(j,j+1)$,  writing for short  $J$ for  $K_{\alpha_j}$ and  $J'$ for $K_{\alpha_{j+1}}$,  the permutation $w_{J,J'}$ is  the  identity in the  case of  move M3,  and different from the  identity  for the  move M2.    Since  $L$ is  the closure of $\alpha$ and  $L'$ is  the  closure of $\alpha'$,  the  product of  all  $w_{J,J'}$  coincides with  the permutation $w_{L,L'}$ operating the iso--conjugation between the combinatoric tied links $(L,J)$ and $(L',J')$.

\end{proof}

  Theorems \ref{Alexander} and \ref{Markov} imply the following
\begin{corollary}
The mapping $\alpha \mapsto \widehat{\alpha}$ defines a bijection between  $T\!B_{\infty}/\sim_{tM}$ and $\mathcal{L}^t/ \sim_t$.
\end{corollary}

\begin{example} \rm  We  show how   the replacement M1 works. Consider the tied braids $(I,\alpha)$,   in Figure \ref{Fig10}, and $(I,\alpha)(\mu_{3,6}, 1)$,  see Figure  \ref{Fig7}. Here $K_\alpha=\{\{1,2\},\{3,6\},\{4\},\{5\}\}$, so  $\pi_\alpha(\mu_{3,6})=\mu_{3,6}$. Clearly,  $\{\{3,6\}\}\preceq K_\alpha$.

\begin{center}
 \begin{figure}[H]
 \includegraphics{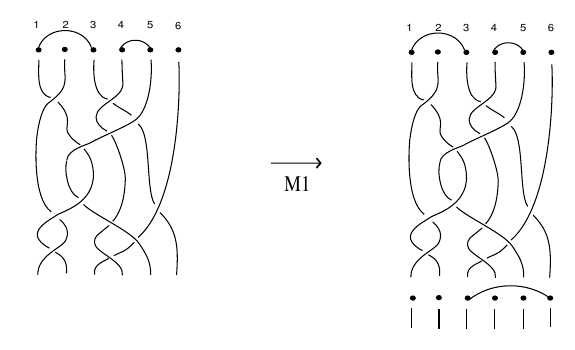}
 \caption{ $(I,\alpha)(\mu_{3,6},1)$ has  the  same  closure as $(I,\alpha)$  shown in  Fig. \ref{Fig10}.}
\label{Fig7}
 \end{figure}
 \end{center}

\end{example}

 \begin{example} \rm
  We  show how   the replacement M2 works. In   Figure  \ref{Fig6a},  we  see  two  braids $\alpha\beta$  and $\beta\alpha$, with $K_{\alpha \beta}=\{\{1,3\},\{2\},\{4\} \}$   and  $  K_{ \beta\alpha } =\{\{1\},\{2,4\},\{3\},  \} $, so  that $\pi_{\alpha\beta}=(1,3)$, $\pi_{\beta\alpha}= (2,4)$.

In  Figure \ref{Fig6}  we  see  the tied braids  $(I_1,	\alpha )$ and $(I_2, \beta)$, with $I_1 = \{\{1,2 \},\{3\},\{4\} \}$ and $I_2= \{ \{1\},  \{2,4\},\{3\} \}$.   Consider   now the  closures of $(I_1,	\alpha )(I_2, \beta)$  and   $(I_2, \beta)(I_1,	 \alpha )$.  These tied links are, respectively, $(\widehat{\alpha\beta}, J_1 )$ and $(\widehat{\beta\alpha}, J_2)$, with  the  sc--partitions  $J_1$  and  $J_2$   given  by
$$
J_1 = \overline {I}_{\alpha\beta}/K_{\alpha\beta} \quad \text{and}\quad
J_2 = \overline {I}_{\beta \alpha}/K_{ \beta \alpha},
$$
where
$$ \overline {I}_{\alpha\beta }:= I_1\ast \pi_{\alpha}(I_2)\ast K_{\alpha\beta} \quad \text{and}\quad \overline {I}_{\beta \alpha}:= I_2\ast \pi_{\beta}(I_1)\ast K_{ \beta\alpha}.
$$
We have
 $\pi_\alpha= (1,4)(2,3)$    and   $\pi_\beta=(1,2,3,4),$  so  that:
$$\pi_\alpha(I_2) = \{\{1,3\},\{2\},\{4\} \}\quad \text{and}\quad \pi_\beta(I_1)  =  \{\{1\},\{2,3 \},\{4\}  \}   ,$$
 and
 $$I_1\ast \pi_\alpha(I_2)\ast  K_{\alpha\beta}=  \{\{1,2,3 \},\{4\}\} ,$$
  $$I_2\ast \pi_\beta(I_1)\ast  K_{ \beta\alpha}=    \{\{1\},\{2,3,4\} \} .$$
Finally,
$$J_1=\{\{1,2,3\},\{4\}  \}/\{\{1,3\}_1,\{2\}_2,\{4\}_3 \}= \{\{1,2\}, \{3\}\},$$
 and
 $$J_2=\{\{1 \},\{2 ,3,4\}  \}/\{\{1 \}_1,\{2,4\}_2,\{3\}_3  \} = \{\{1 \}, \{2, 3\}\}.$$
Observe  now  that  $\pi_\beta(K_{\alpha\beta})=K_{\beta\alpha}$,  and  the corresponding  permutation  of $S_3$ is  $w_{K,K'}=(1,2,3)$. Indeed, $J_2=w_{K,K'}(J_1)$.

\begin{center}
 \begin{figure}[H]
 \includegraphics{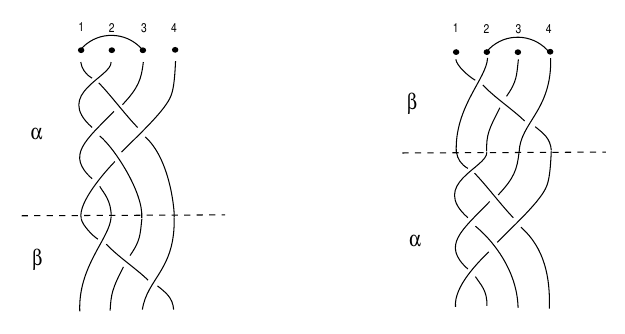}
 \caption{The tied   braids  $(K_{\alpha\beta},\alpha\beta)$  and  $(K_{\beta\alpha},\beta\alpha)$.}
\label{Fig6a}
 \end{figure}
 \end{center}

\begin{center}
 \begin{figure}[H]
 \includegraphics{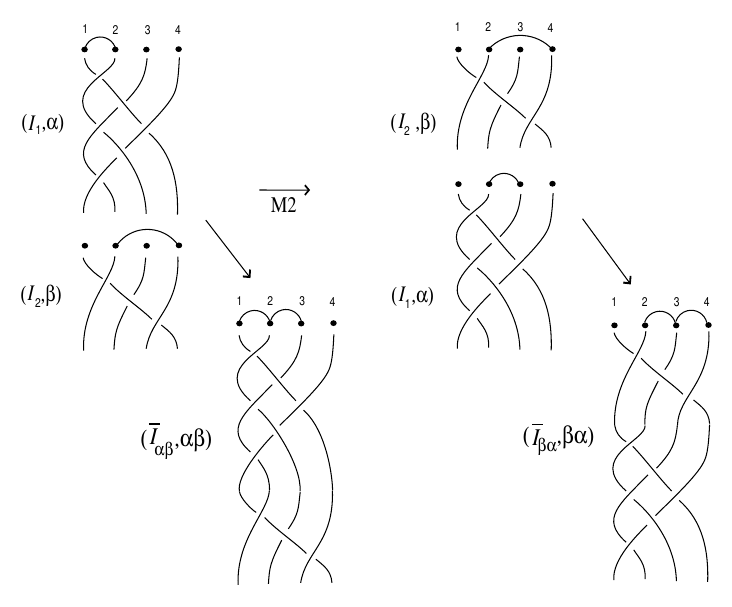}
 \caption{Here  $I_1=\{\{1,2\}\}$  and  $I_2=\{\{2,4\}\}$. }
\label{Fig6}
 \end{figure}
 \end{center}
\end{example}

\section{Invariant for singular links}
  In this section we  define   four families of invariants for singular links constructed by using the Jones recipe applied to the bt--algebra. We discuss also their definitions by skein relations. We start the section with a short recalling of the singular links theory.
\subsection{}

A singular link  is a classical link   admitting simple singular points. Thus,   singular links are a generalization of classical links.    Singular links can be studied trough singular braids: two   singular links are isotopic if their respective singular braids  are Markov equivalents; below we will be more precise.

Let $\SB$ be the {\it singular braid monoid} defined independently by Baez \cite{baLMP}, Birman \cite{biBAMS} and Smolin \cite{smLN}.  $\SB$ is defined by
the elementary braid generators and their inverses $\sigma_1^{\pm 1}, \ldots,
\sigma_{n-1}^{\pm}$ and by  the  elementary singular braid generators $\tau_1, \ldots, \tau_{n-1}$, which  are subjected, besides the
  braid relations among the $\sigma_i$'s, to  the    following relations:
\begin{eqnarray}\label{SB}
\begin{array}{rclcll}
 \tau_i \tau_j & = & \tau_j \tau_i   \qquad \text{for $\vert i  -  j\vert >1$},\\
\sigma_i \tau_i & = &  \tau_i\sigma_i  \qquad \text{for all $i$}, \\
\sigma_i \tau_j & = & \tau_j \sigma_i  \qquad \text{for $\vert i  -  j\vert >1$}, \\
 \sigma_i \sigma_j \tau_i &  = &   \tau_j \sigma_i \sigma_j  \qquad \text{for $\vert i  -  j \vert = 1$}.
\end{array}
\end{eqnarray}
This monoid is the basis for the Alexander theorem and for the  Markov theorem for singular links, which are due, respectively,  to J. Birman \cite{biBAMS} and B. Gemein \cite{geJKTR}.
More precisely, we have the following theorem.
\begin{theorem}\label{MarkovAlexanderSingular}
Every singular link can be obtained as closure of a singular braid and the closures of two singular braids are isotopic singular links if and only if the singular braids can be  obtained one from the other by  means of a finite number of  replacements Ms1 and/or Ms2, where:
\begin{enumerate}
\item[Ms1.] For all $\alpha , \beta \in SB_n$:
$\alpha\beta\quad\text{is replaced by}\quad \beta \alpha,$
\item[Ms2.] For all $\alpha \in SB_n$:
$\alpha \quad \text{replaced by} \quad \alpha \sigma_n  \quad \text{or} \quad \alpha \sigma_n^{-1}$.
\end{enumerate}

\end{theorem}

\subsection{}

In this subsection we  define invariants of singular  links by using the Jones recipe applied to the bt--algebra, that is, the invariants are obtained essentially from the composition $\rho\circ \pi$, where $\pi$ is a representation of $SB_n$
in the bt--algebra and $\rho$ the trace on it, see Theorem \ref{tracebt}.

 Set $\w,\x$ and $\y$  three variable commuting among them and with $\a$ and $\b$. Define $\mathbb{L}$ as the field of rational functions $\mathbb{K}(\a,\b, \x,\y,\w)$. From now on we work on the $\mathbb{L}$--algebra $\mathcal{E}_n(\u)\otimes_{\mathbb{K}}\mathbb{L}$ which is denoted again by $\mathcal{E}_n(\u)$, or simply by  $\mathcal{E}_n.$

\begin{proposition}\label{homopsiphi}
 We have:
\begin{enumerate}
\item The mappings $\sigma_i\mapsto \w T_i$ and $\tau_i\mapsto \x + \y \w T_i$ define a   monoid homomorphism, denoted by $\psi_{n,\w, \x,\y}$,    from $SB_n$ to $\mathcal{E}_n(u)$.
\item The mappings $\sigma_i\mapsto \w T_i$ and $\tau_i\mapsto \x E_i + \y\w E_iT_i$ define a   monoid homomorphism, denoted by $\phi_{n,\w, \x,\y}$, from $SB_n$ to $\mathcal{E}_n(u)$.
  \item The mappings obtained  by  replacing $T_i$ with $V_i$  in items (1) and (2) (see Remark \ref{Vi}),  define two monoid  homomorphisms,  denoted  respectively    $\psi^{\prime}_{n,\w, \x, \y}$ and  $\phi^{\prime}_{n,\w, \x, \y}$, from $SB_n$ to  $\mathcal{E}_n(\v)$.
\end{enumerate}

\end{proposition}
\begin{proof}
  We need to verify that  such mappings  respect  the  defining
   relations of $SB_n$; this checking is a routine and is left to the reader. Notice that the second claim  is  a generalization of  \cite[Proposition 3]{aijuMMJ1}.
\end{proof}
\begin{remark}\label{thomopsiphiprime}\rm
   We will justify later  the distinction apparently superfluous between   $\psi_{n,\w, \x, \y}$ and  $\psi^{\prime}_{n,\w, \x, \y}$ and between  $\phi_{n,\w, \x, \y}$ and  $\phi^{\prime}_{n,\w, \x, \y}$.
\end{remark}

  In order to derive invariants from the homomorphism of Proposition \ref{homopsiphi}, we note that, due to replacement Ms2 of Theorem \ref{MarkovAlexanderSingular},  $\w$ must satisfy  (by using Theorem \ref{tracerho} and   (\ref{Tinverse})):
\begin{equation}\label{defcpsi}
\w^2= \frac{(\rho_n \circ \psi_{n,\w,\x, \y})(\sigma_{n-1}^{-1})}{(\rho_n \circ \psi_{n,\w,\x, \y})(\sigma_{n-1})}= \frac{(\rho_n \circ \phi_{n,\w,\x, \y})(\sigma_{n-1}^{-1})}{(\rho_n \circ \phi_{n,\w,\x, \y})(\sigma_{n-1})}=\frac{\a  + (1-\u)\b}{\a\u}.
\end{equation}
Now, set
$
\c:=\w^2.
$
Then,  for any  singular link  $L$, obtained as the closure of a singular  braid $\omega\in SB_n$,  we  define:
\begin{equation}\label{Psi}
\Psi_{\x, \y}(L) := \left(\frac{1}{\a\sqrt{\c}}\right)^{n-1}
(\rho_n\circ \psi_{n, \sqrt{\c},\x, \y} )(\omega),
\end{equation}

and
\begin{equation}\label{Phi}
\Phi_{\x, \y}(L) :=  \left(\frac{1}{\a\sqrt{\c}}\right)^{n-1}
(\rho_n\circ \phi_{n,\sqrt{\c},\x, \y} )(\omega).
\end{equation}
  Notice that  $\Psi_{\x,\y}$ and $\Phi_{\x,\y}$ take  values in $\mathbb{K}(\a,\x,\y,\sqrt{\c}) = \mathbb{K}(\b ,\x,\y,\sqrt{\c}).$
\begin{theorem}\label{PsiPhi}
The functions $\Psi_{\x, \y}$ and $\Phi_{\x, \y}$ are    ambient isotopy invariants of singular links.
\end{theorem}
\begin{proof}
  We have to   prove that the  functions $\Psi_{\x, \y}$ and $\Phi_{\x, \y}$ respect  the moves Ms1 and Ms2 of Theorem \ref{MarkovAlexanderSingular}.  In  fact, both functions respect Ms1 as consequence  of rule (1) of Theorem \ref{tracerho}, together with fact that $\psi_{n,\sqrt{\c},\x, \y}$ and $\phi_{n,\sqrt{\c},\x, \y}$ are  homomorphisms.

We check now that $\Psi_{\x, \y}(\widehat{\omega\sigma_n^{-1}}) =\Psi_{\x, \y} (\widehat\omega)$, for $\omega\in SB_n$. We have:
\begin{eqnarray*}
(\rho_{n+1}\circ \psi_{n+1, \sqrt{\c},\x, \y} )(\omega\sigma_n^{-1}) =\frac{1}{ \sqrt{\c}}\rho_{n+1} (\omega T_n^{-1}) & = &  \rho_n(\omega)\frac{\a +(\u^{-1}-1)\b + (\u^{-1}-1)\a}{\sqrt{\c}}\\
& = &  \frac{(1-\u)\b + \a}{\u\sqrt{\c}}\rho_n(\omega);
\end{eqnarray*}
hence, $(\rho_{n+1}\circ \psi_{n+1, \sqrt{\c},\x, \y} )(\omega\sigma_n^{-1}) =\a\sqrt{\c}\rho_n(\omega)$. Then,
$$
\Psi_{\x, \y}(\widehat{\omega\sigma_n^{-1}})= \left(\frac{1}{\a\sqrt{\c}}\right)^n
\a\sqrt{\c}\rho_n(\omega) =\Psi_{\x, \y}(\widehat{\omega}).
$$
In the same way we prove  that $\Psi_{\x, \y}(\widehat{\omega\sigma_n})= \Psi_{\x, \y}(\widehat\omega) $. The proof that $\Phi_{\x, \y}$ respect Ms2 is   analogous.
\end{proof}

The invariants $\Psi_{\x, \y}$ and $\Phi_{\x, \y}$ have, respectively, companions $\Psi_{\x, \y}^{\prime}$ and $\Phi_{\x, \y}^{\prime}$, which we   define   now.
   Firstly,  notice that,   because of (3)  Proposition \ref{homopsiphi},   we need in this  case,    by using Theorem \ref{tracerho}:
\begin{equation}\label{defcphi}
\w^2 = \frac{(\rho_n \circ \psi_{n,\w,\x, \y}^{\prime})(\sigma_{n-1}^{-1})}{(\rho_n \circ \psi_{n,\w,\x, \y}^{\prime})(\sigma_{n-1})}=\frac{(\rho_n \circ \phi_{n,\w,\x, \y}^{\prime})(\sigma_{n-1}^{-1})}{(\rho_n \circ \phi_{n,\w,\x, \y}^{\prime})(\sigma_{n-1})}=\frac{\a +(1-\v^2  )\b}{\a },
\end{equation}
  being $\rho_n (V_{n-1})=\a\v^{-1}$,   see  Eq. (\ref{defVi}).   Secondly, define  $\d=\w^2$. Thus,
for $L =\widehat{\omega}$, with $\omega\in SB_n$, we  define:
\begin{equation}\label{Psitilde}
\Psi_{\x, \y}^{\prime}(L) := \left(\frac{  \v }{\a\sqrt{\d}}\right)^{n-1}
(\rho_n\circ \psi_{n, \sqrt{\d}, \x, \y}^{\prime})(\omega),
\end{equation}
and
\begin{equation}\label{Phitilde}
\Phi_{\x, \y}^{\prime}(L) :=  \left(\frac{  \v }{\a\sqrt{\d}}\right)^{n-1}
(\rho_n\circ \phi_{n,\sqrt{\d},\x, \y}^{\prime} )(\omega).
\end{equation}

Notice that   $\Psi_{\x,\y}^{\prime}$ and $\Phi_{\x,\y}^{\prime}$ take  values in $\mathbb{C}(\v,\x,\y,  \a,\sqrt{\d}) = \mathbb{C}(\v, \x,\y,\b,\sqrt{\d}).$

\begin{theorem}\label{tildePsiPhi}
The functions   $\Psi_{\x, \y}^{\prime}$ and $\Phi_{\x, \y}^{\prime}$  are  ambient isotopy invariants of  singular links.
\end{theorem}
\begin{proof}
The same as   the proof of  Theorem \ref{PsiPhi}.
\end{proof}

\begin{remark}\rm
Specializing $\x=\y=0$, the invariant $\Phi_{\x,\y}$  evaluated on classical links coincides with the invariant $\overline{\Delta}$ defined in \cite{aijuMMJ1} and the  invariant $\Phi_{\x,\y}^{\prime}$ coincides with the invariant $\Theta$ defined in \cite{chjukalaIMRN}.
\end{remark}

\begin{remark}\label{Phixx}\rm
 Let $\omega \in   SB_n$ and  $s(\omega)$ the number its singularities. We have:
\begin{enumerate}
\item $(\rho_n \circ \phi_{n,\w,\x,\x})(\omega)= \x^{s(\omega )}(\rho_n \circ \phi_{n,\w,1,1})(\omega )$. Then, $\Phi_{1,1}$ and $\Phi_{\x,\x}$ are equivalent  invariants.
\item $(\rho_n \circ \phi_{n,\w,\x, \y})(\omega)= \x^{s(\omega )}(\rho_n \circ \phi_{n,\w,1, \x^{-1}\y})(\omega )$.  Then, $\Phi_{\x, \y}$ and $\Phi_{1,\x^{-1}\y}$ are equivalent  invariants.  In particular, it follows that  $\Phi_{\x, \y}$ is equivalent to $\Phi_{\tilde \x,\tilde \y}$ if and  only if $\y\x^{-1} =\tilde \y {\tilde \x}^{-1}$.
\end{enumerate}
\end{remark}

\begin{proposition}\label{Phixy}
 The polynomials   $\Phi_{\x,\x}$   and $\Phi_{\x, \y}$  are  not equivalent
 if $\x\not= \y$.
\end{proposition}

\begin{proof}
To prove this proposition, it is  sufficient to  show  a pair of  non  isotopic singular links  which are  are  distinguished by $\Phi_{\x, \y}$ but  not  by  $\Phi_{\x,\x}$. This  in  done  in Section \ref{ctsproof},  Theorem \ref{SS}.
\end{proof}

\begin{remark}\rm   We show  now    how  the  invariant $\Phi_{\x, \y}$  generalizes the invariant  $\bar \Gamma$ defined in \cite{aijuMMJ1}.
Writing $\omega =\omega_1^{\epsilon_1}\cdots \omega_m^{\epsilon_m}$, where the
$\omega_i$'s are the defining generators   of $\SB$, we define  the exponent $\epsilon(\omega)$ of $\omega \in \SB$ as
$$
\epsilon(\omega) := c_1\epsilon_1 +c_2\epsilon_2 + \ldots + c_m\epsilon_m,
$$
where  $c_i=1$ if  $\omega_i=\sigma_i^{\pm 1}$, whereas  $c_i=0$ if  $\omega_i=\tau_i$. Then, the invariant   $\Phi_{1, 1/\w}$  can be written as follows:
$$
\Phi_{1,1/\w} (L) =  \left(\frac{1}{\a\sqrt{\c}}\right)^{n-1} \sqrt{\c}^{\;\epsilon (\omega)}(\rho_n\circ \phi_{n,1,1,1} )(\omega),
$$
where $L=\widehat{\omega}$.   On the other hand,  the invariant $\overline{\Gamma}$   can be written as:
$$
\overline{\Gamma} (L) =  \left(\frac{1}{\a\sqrt{\c}}\right)^{n-1} \sqrt{\c}^{\;\epsilon (\omega)+s(L)}(\rho_n\circ \phi_{n,1,1,1} )(\omega).
$$
  Hence,
$$
\overline{\Gamma} (L) =\sqrt{\c}^{\;s(L)}  \Phi_{1,1/\w} (L) ,
$$
where $L=\widehat{\omega}$ and $s(L)$ denotes the number of singular points of $L$.
 The exponent $s(L) $ is needed  since  in \cite{aijuMMJ1} and   \cite{julaJKTR} the definition of the  exponent  $\omega$   takes  $c_i=1$ when  $\omega_i=\tau_i$.
 \end{remark}
\section{Tied singular links}\label{ctssection}
 In this section  we   introduce the   tied singular
links  and the combinatoric   tied singular links. We introduce also the
  monoid of tied singular braids. The section ends by proving the Alexander and Markov theorems for tied singular links.

\subsection{}

We have two natural monoid homomorphisms from $SB_n$ onto $B_n$:
the first one, denoted by $f$, maps   $\sigma_i$ to $\sigma_i$ and  $\tau_i$ to $\sigma_i$ and the second one, denoted by $f^{-}$, maps   $\sigma_i$ to $\sigma_i$ and   $\tau_i$ to $\sigma_i^{-1}$;
notice that for every $\omega\in SB_n$, $(\pi\circ f)(\omega)= (\pi\circ f^-)(\omega)$, where $\pi$, as in Subsection 1.2, denote the natural epimorphism from $B_n$ to $S_n$.
Let $L$ be a singular link   obtained as the closure of $\omega\in SB_n$; the closure, respectively,  of $f(\omega)$ or $f^-(\omega)$ is the classical links obtained by replacing  every singular point of $L$ by a positive crossing or, respectively, by a negative crossing.   In  terms  of  singular links,  the replacement of  a  singular point  by  a  positive or a  negative  crossing is  called  {\it simple  desingularization},  and  the  singularity  is  said {\it  simply  desingularized}.

\begin{definition} The number of components of a singular link  $L$, closure of  a  singular braid $\omega$,   is the number of disjoint cycles of   $(\pi \circ f)(\omega)$. In other words,
 the number of components of  a singular link $L$ is   the number of components of the classical link obtained by replacing  every singular crossing with a positive (or negative) crossing in $L$.
 \end{definition}

Let $\mathcal{L}^s$ be the set  of  isotopy  classes  of  singular  links in $\mathbb{R}^3$,   and $\mathcal{L}^s_{k,m}$   the set  formed by those   with $k$ components  and $m$ singularities. Thus

$$
\mathcal{L}^s = \coprod_{k>0,m\geq 0}
\mathcal{L}^s_{k,m}.
$$
The elements of $\mathcal{L}^s_{k,m}$ are called $(k, m)$--singular links.

\begin{definition}
A tied singular link is a singular link with ties s.t. whenever
the singular points are simply desingularized one  obtains a tied link.
\end{definition}
\begin{definition}\label{cts}
Set $\mathcal{L}_{k,m}^{t,s} := \mathcal{L}^s_{k,m}\times \mathsf{P}_{k}$. The elements of $\mathcal{L}_{k,m}^{t,s}$ are called $(k, m)$--combinatoric tied singular links. We call   combinatoric tied singular links  (for short cts--links) the elements  of $\mathcal{L}^{t,s}$, where
$$
\mathcal{L}^{ t,s} :=\coprod_{k>0, m\geq 0}
\mathcal{L}_{k,m}^{t,s}.
$$
\end{definition}

\subsection{}
The group $B_n$ is  naturally a submonoid of $SB_n$ and the
natural epimorphism $\pi :B_n \rightarrow S_n$,  defined in Subsection 1.2, can be extend to   $SB_n$ by mapping $\tau_i$ to $s_i$. We denote again this extension by $\pi$ and, consequently, we denote the image of $\tau_i$ by $\pi$ by $\pi_{\tau_i}$.

 As before we define a monoid structure on the cartesian product $\Psf\times\SB$, cf. (\ref{mult}), as follows: $ (I, \alpha  ) (J, \beta ) = (I\ast \pi_\alpha(J), \alpha\beta )$, where $I, J\in \Psf$ and $\alpha ,\beta\in \SB$;   we denote this monoid by $\Psf \rtimes\SB$.

Notice that the  elements $\widetilde{\mu}_{i,j}$'s can be considered in $\Psf \rtimes\SB$.  We have:
$$
(1, \tau_k)(\mu_{i,j}, 1) = (\pi_{\tau_k}(\mu_{i,j}),\tau_k) = (\pi_{\sigma_k}(\mu_{i, j}),1)(1, \tau_k).
$$
Since Eq. (\ref{sigmaI}) holds in $\Psf \rtimes\SB$,   we conclude that
\begin{equation}\label{KeyRelationTSB}
\tau_k\widetilde{\mu}_{i,j} = \sigma_k\widetilde{\mu}_{i,j}\sigma_k^{-1}\tau_k.
\end{equation}

\begin{definition}\label{monotsb}
 We define  $\TSB$ as the monoid  presented by the  braid generators  $\sigma_1^{\pm 1},$ $ \ldots ,$ $\sigma_{n-1}^{\pm 1}$, the singular braid generators $\tau_1, \ldots , \tau_{n-1}$ of $\SB$ and the
 ties generators $\eta_1, \ldots ,\eta_{n-1}$ of $T\!B_n$, subject to the following relations:
  the defining relations of $SB_n$, the defining relations of $\TB$ and the relations:
\begin{eqnarray}
\tau_i \eta_i & = &   \eta_i \tau_i \qquad \text{ for all $i$},
\label{TSB1}\\
\tau_i\eta_j  & = &   \eta_j\tau_i  \qquad \text{for $\vert i  -  j\vert >1$},
\label{TSB2}\\
\eta_i \tau_j \tau_i & = &   \tau_j \tau_i \eta_j \qquad
\text{ for $\vert i - j\vert =1$},
\label{TSB3}\\
\eta_i\eta_j\tau_i & = & \eta_j\tau_i\eta_j  \quad = \quad\tau_i\eta_i\eta_j \qquad \text{ for $\vert i  -  j\vert =1$},
\label{TSB4}\\
\eta_i\tau_j\sigma_i  & = &   \tau_j\sigma_i \eta_j \qquad \text{ for  $\vert i  -  j\vert =1$},
\label{TSB5}\\
\eta_i \sigma_j \tau_i & = & \sigma_j \tau_i \eta_j \qquad \text{ for  $\vert i  -  j\vert =1$},
\label{TSB6}\\
\tau_i\eta_j & = & \sigma_i\eta_j\sigma_i^{-1} \tau_i \qquad \text{ for  $\vert i  -  j\vert =1$}.
\label{TSB7}
\end{eqnarray}
\end{definition}

\begin{remark}\rm
Note that, by allowing $i,j$ take  every possibility  in (\ref{TSB7}),   the relations (\ref{TSB1}) and (\ref{TSB2}) are included  in (\ref{TSB7}). Further,   relation (\ref{TSB7}) can be written, equivalently, by exchanging $\sigma_i$ with $\sigma_i^{-1}$.
\end{remark}
\begin{theorem}\label{wreathTSB}
The monoids  $\TSB$ and  $\Psf \rtimes\SB$ are isomorphic.
\end{theorem}
\begin{proof}(Analogous to proof of Theorem \ref{TBnsemidirect})
It is a routine to check that the mappings $\sigma_i\mapsto\sigma_i$, $ \tau_i\mapsto\tau_i$ and $\eta_i\mapsto \widetilde{\mu}_{i, i+1}$ define  a morphism  $\phi$  from
$\TSB$ to $\Psf \rtimes\SB$. Now, arguing as in Lemma \ref{generators}, we obtain that $\Psf \rtimes SB_n$ is generated by $\widetilde{\mu}_{1,2}, \ldots ,\widetilde{\mu}_{n-1, n}$, $\sigma_1^{\pm 1},\ldots ,\sigma_{n-1}^{\pm 1}$, $\tau_1, \ldots ,\tau_{n-1}$; then it follows that $\phi$ is an epimorphism. The  injectivity of $\phi$ is proved as the
 injectivity in Theorem \ref{relationsPn};  therefore   it is enough to prove  that every element $\Psf \rtimes SB_n$
 has the decomposition $\eta_I\beta$, where $I\in \Psf$ and $\beta\in SB_n$; in the present situation such decomposition is obtained by combining (4) of Lemma \ref{lemmaEtas} with (\ref{KeyRelationTSB}),  cf. Proposition \ref{decomposition}.
\end{proof}

\begin{definition}\label{closuretsb}
 The closure of a {\it tied singular braid}, i.e. an element of $\TSB\simeq \Psf \rtimes\SB$, is defined  in the same way as   the closure  of a tied braid (see Definition \ref{closure}).
I.e., given a tied singular braid $(J, \omega)$,  its closure
$\widehat{(J, \omega)}$ is equal  to $(\widehat{\omega}, I)$, with $I=I_{f(\omega)}$, where  $I_{f(\omega)}$ is the set partition defined by the closure of  the tied braid $(J, f(\omega))$.
\end{definition}

\subsection{}
We are now in position to establish and to prove the Alexander and Markov theorem for cts--links.

\begin{theorem}[Alexander theorem for cts--links]\label{AlexanderTSL}
Every  cts--link can be obtained as the closure of a tied singular braid. More precisely, given the cts--link $(L, J)$, where  $L$ is the closure of the  singular braid $\omega$, then $(L, J)$, up to a renumbering the components, is the closure of the tied singular braid $(I, \omega)$, where  $I$ is defined by
$$
I= K_{f(\omega)}\times J.
$$
\end{theorem}
\begin{proof}
Set $(L, J)\in \mathcal{L}^s_{k,m}\times \mathsf{P}_k$. Let $\omega=\omega_1\cdots \omega_l\in SB_n$ whose closure is $L$, where the $\omega_i$'s are the defining generators of $SB_n$.  Let $L^{\prime}$ be the classical link obtained as the closure of $f(\omega)$; so $k$ is the number of components of $L^{\prime}$. Now, from Theorem \ref{Alexander} applied to the combinatoric tied link $(L^{\prime},J)$, we have that, up to the renumbering the components, it is the closure of the tied braid $(I, f(\omega))$, where $I= K_{f(\omega)}\times J$. Thus, it  follows that $(L, J)$ is the closure of the tied singular braid $(I, \omega)$.
\end{proof}

\begin{theorem}[Markov theorem for cts--links]\label{MarkovTSL} Two
tied singular braids yield  the same  cts--link if and  only if   they are $\sim_{Mts}$--equivalent, i.e.  one   is obtained from the other by using the replacements Mts1/Mts2/and or Mts3 below.
\begin{enumerate}
\item[Mts1.] t--Stabilization: for all $(I,\alpha)\in TSB_n$, we can do the following replacements:
$$
(I,\alpha ) \quad \text{replaced by}\quad   (I,\alpha)(\mu_{i,j},1) \quad \text{if} \quad    i,j \quad  \text{belong to the  same  cycle of} \quad  \pi_\alpha ,
$$
\item[Mts2.] Commuting in $TSB_n$: for all $(I_1,\alpha), (I_2,\beta)\in TB_n$, we can do the following replacement:
$$
(I_1,\alpha)(I_2,\beta )\quad\text{replaced by}\quad (I_2,\beta ) (I_1,\alpha)  ,
$$
\item[Mts3.] Stabilizations: for all $(I,\alpha) \in TSB_n$, we can do the following replacements:
$$
 (I,\alpha) \quad \text{replaced by}\quad(I,\alpha \sigma_n ) \quad \text{or} \quad(I,\alpha \sigma_n^{-1}  ) .
$$
\end{enumerate}
\end{theorem}
\begin{proof}

  Let  $(I, \omega)\in TSB_n$; according to Definition \ref{closuretsb}, the set partition determined by  $\widehat{(I, \omega)}$ is the set partition determined by
 $\widehat{(I, f(\omega))}$. Thus, the verification that the replacements Mts1, Mts2 and/or Mts3 do  not alter the closure of a singular tied braid results in a repetition  of the verification that  the Markov replacements  of   Theorem \ref{Markov} do  not affect the closure  of a combinatoric tied braid.

 In the other direction we use again the fact that, by definition, the set partition determined by $\widehat{(I, \omega)}$ is equal to the set partition determined by $\widehat{(I, f(\omega))}$. Then the  proof  follows from  those of  Theorem \ref{MarkovAlexanderSingular} and Theorem \ref{Markov}.

\end{proof}

\section{Invariants of cts--links}\label{ctsinv}
We will extend the invariants of Theorems \ref{PsiPhi} and \ref{tildePsiPhi} to invariants for cts--links.
Thanks to Theorems \ref{AlexanderTSL} and \ref{MarkovTSL}, we deduce that the  definitions of  these extensions are reduced to extend the domain of the maps $\phi_{n,\w,\x, \y}$ and $\psi_{n,\w,\x, \y}$ to $TSB_n$. More generally, the next  proposition extends the homomorphisms of Proposition \ref{homopsiphi}.

\begin{proposition}\label{thomopsiphi}
For all $n$,  the domain of definition of the morphisms  $\phi_{n,\w,\x, \y}$ and $\psi_{n,\w,\x, \y}$ can be extended  to $\TSB$, by mapping  $\eta_i$ to  $E_i$. We shall keep, respectively,  the same notations  $\phi_{n,\w,\x, \y}$ and $\psi_{n,\w,\x, \y}$  for  these extensions.
\end{proposition}
\begin{proof}
The proof follows by checking that these extensions respect  the relations (\ref{TSB1})--(\ref{TSB7}); these checkings are straightforward. For instance, we now verify Relation (\ref{TSB5}). Set $\phi_n=\phi_{n,\w,\x, \y}$ and suppose $\vert i -j\vert =1$. Since    $E_iE_j$ commutes
with $T_i$ and $T_j$, we have:
$\phi_n(\eta_i)\phi_n(\tau_j)\phi_n(\sigma_i)=
E_i(\x E_j+\y\w E_jT_j)\w T_i = (\x E_j+\y\w E_jT_j)\w T_i E_i$, from (\ref{bt6}). So:
$$
\phi_n(\eta_i)\phi_n(\tau_j)\phi_n(\sigma_i)= (\x E_j+\y\w  E_jT_j)\w T_i E_i =\phi_n(\tau_j)\phi_n(\sigma_i)\phi_n(\eta_j).
$$
\end{proof}
\begin{remark}\label{tthomopsiphi}\rm
 Also the domain of the homomorphisms (3) of  Proposition \ref{homopsiphi}   can be extended to $TSB_n$.  As  in  the proposition above, we keep the same notations, that is $\psi_{n,\w, \x, \y}^{\prime}$ and   $\phi_{n,\w, \x, \y}^{\prime}$, for these extensions.
\end{remark}

The invariants $\Psi_{\x, \y}$ and $\Phi_{\x, \y}$ for singular links can be extended to invariants for tied singular links simply by taking, respectively, in the definition of  (\ref{Psi}) and (\ref{Phi}), the extensions  $\psi_{n,\sqrt{\c}, \x, \y}$ and $\phi_{n,\sqrt{\c}, \x, \y}$ to $TSB_n$ of Proposition \ref{thomopsiphi}; we denote again these invariants for  cts--links by
$\Psi_{\x, \y}$ and $\Phi_{\x, \y}$. Repeating the argument on the invariants of Theorem \ref{tildePsiPhi}, we obtain invariants for cts--links  again keeping  the  notation $\Psi_{\x, \y}^{\prime}$ and $\Phi_{\x, \y}^{\prime}$.

\subsection{Skein rules}\label{ctsskein}

In this  section  we  will   define  the  invariants $\Phi_{\x, \y}$  and  $\Psi_{\x, \y}$    by  skein  rule  and  desingularization.  Recall  now that  both $\Phi_{\x, \y}$  and  $\Psi_{\x, \y}$   are  extensions of  $\F$ to  singular links, therefore the skein rules of them must  contain the  skein rules  of  $\F$; so, in particular,  the defining skein relations of  $\F$ will be reformulated
  in  the  context of  cts--links.
 Recall  that  in a  cts--link $(L,J)$, the components of  $L$ are numbered and  the  parts of  the  set partition $J$  are  standardly  indexed.   Now  we need to introduce the notations below.
\begin{notation}\label{LPM} \rm Consider a  generic  diagram  of a cts--link $(L,J)$, suppose  that  $J$ has   blocks  $J_1,\ldots,J_m$  and $L$  has a  positive crossing  such that the components of this crossing belong     to  two    blocks  $J_i$ and  $J_k$ ($i\le k$) of  $J$.    We  shall  denote  by:
  \begin{enumerate}
\item  $(L_+^{i,k},J) $ the link  $(L,J)$;
\item    $(L_-^{i,k},J )$ the same  as the previous,  but  the  positive crossing is  replaced  by  a  negative  crossing;
\item    $(L_\times^{i,k},J )$ the same  as above,  but now the  crossing is  replaced  by  a  singular crossing;
\item    $(L_+^{i,i},  J' ) $ as $(L_+^{i,k},J) $,   where  $J'$ is the set partition obtained from $J$ by considering  the union of $J_i$  and  $J_k$  as a unique part;
\item    $(L_-^{i,i}, J' ) $ the same  as the previous,  but  the  positive crossing is  replaced  by  a  negative  crossing;
\item    $(L_\times^{i,i},J' ) $ the same  as the previous,  but  the    crossing is  now a  singular  crossing;
\item     $(L_0^{i,i}, J'' )$ the initial  link,  where  the  crossing strands are replaced by  two  non crossing strands and  the  parts  containing  the  components crossing   merge in a  unique part in $J''$.
    \end{enumerate}
\end{notation}

\begin{remark}\rm   Let us suppose $J\in \PP_n$ has $m$ blocks, we have:
\begin{enumerate}
\item  If the crossing components  belong to different blocks,  then    $J'$ is in $\PP_n$ and  has $(m-1)$   blocks.
 Moreover,   the  two  crossing components     merge in a  unique  component in $L_0^{i,i}$,    therefore  $J''\in \PP_{n-1}$.
\item  In the case that the  crossing components   belong to a same block of the set partition, we have  $J'=J$.
   However,  observe  that  the  two    crossing strands   may  belong to  two  different components   or to  a  same  component. In the  first  case, the  two  different components  merge in a  unique component  in     $ L_0^{i,i}$,     then    $J''\in \PP_{n-1}$  and  has $m$ parts;  in the  second  case, the  component  splits in two components,    still belonging  to  the  same  $i$--th   block, thus,   $J''\in \PP_{n+1}$.

\item  Observe that, in order to  define  the skein rules,  neither the  total number of  components nor  the  total  number  of  parts of  the partition  is   relevant.  Therefore
we shall  use  the  notation  $(i,k)$  for  both  cases  $J_i=J_k$  and  $J_i \not=J_k$.
\end{enumerate}
\end{remark}

Before of stating  the main theorems of this section
we introduce  the notation ${\rm I}(L,J)$ to indicate the value  of the invariant $\rm I$ on the cts--link  $(L,J)$.

\begin{theorem}\label{theoremPhi}   The  invariant   $\Phi_{\x, \y}$ of  singular tied links is  defined   uniquely  by  four rules.   More precisely,  the values of $\Phi_{\x, \y}$ on a cts--link $(L,J)$, with $n$ components, is determined through  the rules:

 \begin{enumerate}
\item[I] The  value of  $\Phi_{\x, \y}$ is  equal to 1  on  the  unknotted circle.

\item[II]
$$
\Phi_{\x, \y}(L \sqcup O, \iota_n(J))= \frac{1}{ \a\sqrt{\c}} \Phi_{\x, \y}(L,J),
$$
 where $\iota_n$ is the natural inclusion of $\PP_n$ into $\PP_{n+1}$ (see Definition \ref{Pinf}).
\item[III]  Skein rule.
$$
\frac{1}{\sqrt \c} \Phi_{\x, \y}(L_+^{i,k},J)+  \sqrt \c \Phi_{\x, \y}(L_-^{i,k},J)=  \frac{1}{\sqrt \c} (1-\u^{-1}) \Phi_{\x, \y}(L_+^{i,i},J') + (1-\u^{-1})  \Phi_{\x, \y} (L_0^{i,i}, J'').
$$
\item[IV] Desingularization.
$$
  \Phi_{\x, \y}(L_\times^{i,k},J)= \x  \Phi_{\x, \y} (L_0^{i,i}, J'' )   +  \y    \Phi_{\x, \y}(L_+^{i,i},J')  .
$$
\end{enumerate}
\end{theorem}

\begin{theorem}\label{theoremPsi}
The  invariant $\Psi_{\x, \y}$ is  defined  by  the  same  rules {\rm  I--III }  as $\Phi_{\x, \y}$ in Theorem \ref{theoremPhi} but the desingularization rule {\rm  IV } is  replaced  by
\begin{enumerate}
\item[IV']
$$
  \Psi_{\x, \y}(L_\times^{i,k},J)= \x \Psi_{\x, \y} (L_0^{i,i}, J'' )+     \y    \Psi_{\x, \y}(L_+^{i,k},J) .
$$
\end{enumerate}
 \end{theorem}

\begin{proof} [Proof of  Theorems \ref{theoremPhi} and \ref{theoremPsi}]  \rm  For  non singular combinatoric tied links, both  $\Phi_{\x, \y}$ and  $\Psi_{\x, \y}$  coincide with  the polynomial  $\F$  for  tied links,  defined  in
\cite[Theorem 2.1]{aijuJKTR1}; indeed,   rules I--III  are  exactly  the skein  rules I--III of  $\F$, under the  replacements $\u  \rightarrow u$,  $\sqrt \c \rightarrow  w$, $\a \rightarrow z$,  and  observing that the
  translation between the notations of tied links \cite[Fig. 3]{aijuJKTR1} and cts--links of Notation \ref{LPM} is as follows: the   tied link    $L_\pm$ with  a positive/negative crossing corresponds to   the cts--link   $(L_\pm^{i,k},J) $; $L_{\pm,\sim}$ corresponds  to     $  (L_{\pm}^{i,i},  J' )$   and  $L_{0,\sim}$ corresponds to  $ (L_0^{i,i},J'')$. To conclude the the proof, it remains to verify    the  desingularization rules IV and  IV'.   Suppose  that   the  cts--link $(L_\times^{i,j},  J )$  has only one singularitiy,  and  that it   is the  closure of the  singular  tied  braid $\omega=\alpha \tau_i \beta$,   with $\alpha,\beta \in  \TB $.   In order  to  calculate $\Phi_{\x, \y}$ (respectively  $\Psi_{\x, \y}$), we  have  to  calculate  the trace of the  image of $\omega$ in the  bt--algebra.    By  using Proposition \ref{homopsiphi}, we  obtain that the image of  $\omega$ splits into a  linear combination of  two  elements,  precisely
$$
\x ( \phi_{n,\w, \x, \y}(\alpha)  E_i   \phi_{n,\w, \x, \y}(\beta) ) + \y (\phi_{n,\w, \x, \y}(\alpha) \w  E_i  T_i  \phi_{n,\w, \x, \y}(\beta) ),
$$
  (respectively,
  $  \x (\psi_{n,\w , \x, \y}(\alpha)  \psi_{n,\w , \x, \y}(\beta)) + \y (\psi_{n,\w , \x, \y}(\alpha)  \w T_i  \psi_{n,\w   , \x, \y}(\beta))).$
   These  elements  are  the images in  the bt--algebra of  $\alpha\eta_i \beta$ and $\alpha \sigma_i \eta_i \beta$
(respectively, of $\alpha  \beta$ and $\alpha \sigma_i   \beta$),  whose  closures give  the cts--links $ (L_0^{i,i},  J'') $ and $( L_+^{i,i},  J' )$, (respectively the cts--links  $ (L_0^{i,i},  J'' )$ and $ (L_+^{i,j},  J) $.  The desingularization rules  IV  and IV' then  follow from the  linearity of  the  trace together  with  the defining formulae (\ref{Phi})  and  (\ref{Psi}). If  the  number of  singularities  of the  cts--link is  higher,  say $m$,  the argument  remains  the  same, i.e., by comparing  the result of  the  desingularization  rule IV (or IV') to  all $m$ singularities of  the  link (result that is  independent from  the order on which  they  are applied) and  the image  in the  bt--algebra    of the  corresponding singular braid with $m$  elements  $\tau_i$, according to   the  respective map of  Proposition \ref{homopsiphi}.
\end{proof}

The analogous of Theorems \ref{theoremPhi} and \ref{theoremPsi} are as follows.

\begin{theorem}\label{theoremPhiPsiprime} { \quad \quad}
  \begin{enumerate}
\item
The  invariant $\Phi'_{\x, \y}$ is  defined  by  the  same  rules {\rm  I, II} and   {\rm IV }  as $\Phi_{\x, \y}$ in Theorem \ref{theoremPhi} but the skein rule {\rm  III } is  replaced   now   by:
\begin{enumerate}
\item[III']   Skein rule,
$$
\frac{1}{\sqrt \d} \Phi'_{\x, \y}(L_+^{i,k},J)+  \sqrt \d \Phi'_{\x, \y}(L_-^{i,k},J)=    (\v-\v^{-1})  \Phi'_{\x, \y} (L_0^{i,i}, J'').
$$
\end{enumerate}
\item
The  invariant $\Psi'_{\x, \y}$ is  defined  by  the  same  rules {\rm  I, II} and  {\rm IV' }  as $\Psi_{\x, \y}$ in Theorem \ref{theoremPsi} but the skein rule {\rm  III } is  replaced   by
 the  {\rm  III'}  above in which  $\Phi'$ is replaced  by $\Psi'$.
\end{enumerate}
 \end{theorem}
\begin{proof} \rm  For  non singular combinatoric tied links, both  $\Phi'_{\x, \y}$ and  $\Psi'_{\x, \y}$  coincide with  the polynomial  $\overline \Theta$  for  tied links,  defined  in
 \cite{chjukalaIMRN}; indeed,   rules I--III'  are  exactly  the skein  rules of  $\overline \Theta$, under the  replacements $\v \rightarrow q$,  $\c \rightarrow  \lambda$, $\a \rightarrow z$. After,  the proof proceeds as the proof of Theorems  \ref{theoremPhi}  and \ref{theoremPsi}.
 \end{proof}

\begin{remark} \label{desing}\rm
 \begin{enumerate}
 \item  The  desingularization rules  IV   and   IV'   coincide  when  the components crossing at the  singular point belong to  the  same  part of the  partition.  This  implies that  $\Phi_{\x, \y}=\Psi_{\x, \y}$ for knots  and  cts--links  having a  set  partition  with  a  sole  part.
 \item The  invariants  $\Phi$ and $\Phi'$  have  the  same desingularization rule IV,  while he  invariants  $\Psi$ and $\Psi'$  have  the  same desingularization rule IV'.
\item From the   desingularization rules  IV   and   IV'  it follows that  the    invariant polynomial $\Phi_{\x, \y}$  as well  as  the other invariants $\Phi'_{\x, \y},\Psi_{\x, \y}$  and $\Psi'_{\x, \y}$,   when  evaluated on  a  cts--link  $\Sing$ with  $m$  singularities, is  homogeneous of degree $m$ in  the  variables $\x,\y$ (see  also  Remark \ref{Phixx}):
 $$
 \Phi_{\x,\y}(\Sing) =  \x^m \Phi_{1,\y/\x} (\Sing).
 $$
 \end{enumerate}
  \end{remark}

\section{Comparison of  invariants}

In  this   section  we  compare  the  invariants  here introduced with each other and  with   the Paris--Rabenda invariant \cite{paraALIF}.   The comparison of our invariants is done  on pairs of singular links constructed
from pairs of non isotopic classical links not distinguished  by the Hompflyt
polynomial; these pairs are taken from  \cite{chli}.

\subsection{Notation and some elementary facts.}
In what follows we  will  denote by:
\begin{enumerate}
\item $P$    the  Homflypt  polynomial for  classical  links,
\item $\F$    the polynomial for tied  links,
\item
  $\F'$  the polynomial for tied links   $\overline \Theta$  (see \cite{chjukalaIMRN})  constructed by the Jones recipe as $\F$ but using the presentation of the bt--algebra $\E_n(\v)$  (see Remark  \ref{Vi}) instead of the presentation $\E_n(\u)$,
\item $ \I_{PR}$   the  polynomial  for singular links  due to Paris and Rabenda \cite{paraALIF},
  \item        When   we  say  that   two  non  isotopic  (singular) links   with $n$ components   $L$  and $L'$ are distinguished  by an invariant $\I$  for  tied (singular) links,   we  mean   that $\I(L,\1_n)\not=\I(L^{\prime},\1_n)$.
    In the  same  way we did  in   \cite[Subsection 2.3]{aijuMathZ}.
\end{enumerate}

  The following proposition comes out   by an appropriate renaming of the variables.

\begin{proposition}\label{valuesinv}
\begin{enumerate}
\item If $L$  is  a classical  link,  then  $\I_{PR}(L)=P(L)$,
\item If $L$ is a   classical  link   and $I$    the set partition    with a  unique  block,  then:
$$
\F(L,I)=\F'(L,I)=P(L),
$$
\item If $L$ is a  non singular tied link  with $n$ components, then:
 $$
\Phi_{\x, \y}(L,\1_n)=\Psi_{\x, \y}(L,\1_n)=\F(L,\1_n),
$$
 \item If $L$ is a  non singular tied link  with $n$ components, then:
 $$
\Phi_{\x, \y}^{\prime}(L,\1_n)=\Psi_{\x, \y}^{\prime}(L,\1_n)
=\F^{\prime}(L,\1_n),
$$
 \item If $L$ is a singular classical link   and $I$    the set partition    with a  unique  block,  then
 $$
\Psi_{\x, \y}(L,I)=\Psi'_{\x, \y}(L,I)=\Phi_{\x, \y}(L,I)=\Phi'_{\x, \y}(L,I)=\I_{PR}(L).
$$

\end{enumerate}
\end{proposition}

\subsection{Differences between   $\Phi_{\x, \y}$  and  $\Psi_{\x, \y}$} \label{ctsproof}
   In  this  section  we  analyze some  properties of $\Phi_{\x, \y}$  and $\Psi_{\x, \y}$.   By  Remark \ref{desing} (2),    the next proposition and     Theorem \ref{SS}  hold identically  if  $\Phi_{\x, \y}$  and $\Psi_{\x, \y}$  are  replaced,  respectively, by $\Phi'_{\x, \y}$  and $\Psi'_{\x, \y}$.

The following proposition shows  that  $\Psi_{\x, \y}$ is  more powerful  than  $\Phi_{\x, \y}$ on  cts--links.

 Take any  classical  singular link  $\Sing$ with $n$ components,  having  at  least one singularity  involving  two  distinct  components $i$  and $j$.  Consider  the cts--links  $(\Sing_\times^{i,j},\1_n)$  and $(\Sing_\times^{i,i},\{\{i,j\}\})$.
 \begin{proposition} $$ \Psi_{\x, \y}(\Sing_\times^{i,j},\1_n) \not= \Psi_{\x, \y}(\Sing_\times^{i,i},\{\{i,j\}\}), $$
 while
  $$ \Phi_{\x, \y}(\Sing_\times^{i,j},\1_n)  = \Phi_{\x, \y}(\Sing_\times^{i,i},\{\{i,j\}\}). $$
 \end{proposition}

 \begin{proof} We  have, respectively,  by  rules  IV'  and IV:
 $$ \Psi_{\x, \y}(\Sing_\times^{i,j},\1_n)  =\x\ \Psi_{\x, \y}(\Sing_0^{i,i},\1_{n-1})+ \y \ \Psi_{\x, \y}(\Sing_+^{i,j},\1_{n}),$$
 $$\Psi_{\x, \y}(\Sing_\times^{i,i},\{\{i,j\}\})= \x\ \Psi_{\x, \y}(\Sing_0^{i,i},\1_{n-1})+\y \ \Psi_{\x, \y}(\Sing_+^{i,i},\{\{i,j\}\});$$
 while
 $$ \Phi_{\x, \y}(\Sing_\times^{i,j},\1_n)  =\Phi_{\x, \y}(\Sing_\times^{i,i},\{\{i,j\}\})= \x\ \Phi_{\x, \y}(\Sing_0^{i,i},\1_{n-1})+\y \ \Phi_{\x, \y}(\Sing_+^{i,i},\{\{i,j\}\}).$$
 \end{proof}

 Here  we  show  an  example proving that $\Phi_{\x, \y}$ is not  equivalent  to  $\Phi_{\x,\x}$,   according to  Proposition \ref{Phixy}.  The  same  example allows us to  prove  that $\Psi_{\x, \x}$ distinguishes pairs not  distinguished by $\Phi_{\x, \x}$.

Because of item (3) of Proposition \ref{valuesinv}, the  values of $\Phi_{\x, \y}$  and  $\Psi_{\x, \y}$  coincide  on classical  knots.

 \begin{center}
 \begin{figure}[H]
 \includegraphics{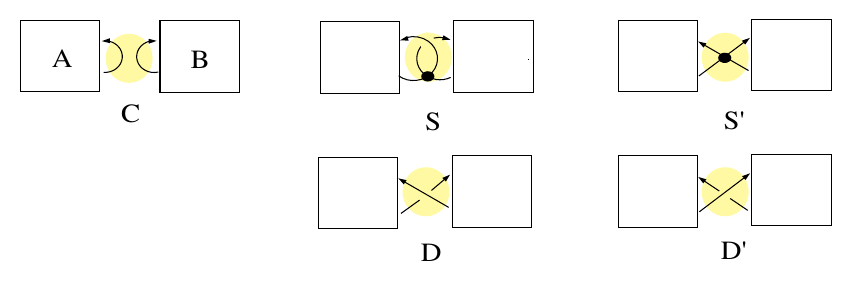}
 \caption{Two singular links $\Sing$ and  $\Sing'$  and  links  involved  in their desingularization.}
\label{exa1}
 \end{figure}
 \end{center}

Take  a link diagram $\C$ made by  two disjoint knots diagrams $\A$  and  $\B$ as shown  in Figure \ref{exa1}.  Then  consider  the  singular  links  $\Sing$  and  $\Sing'$  in  the same  figure,  obtained  by modifying the link $\C$ only  in  the  yellow disk. Evidently  $\Sing$ and  $\Sing'$  are  not  isotopic,  since $\Sing$ has  two components,  while $\Sing'$ is  a knot.

\begin{theorem} \label{SS} The  singular  links  $\Sing$  and  $\Sing'$  are  distinguished by   $\Phi_{\x, \y}$  if  and only if  $\x\not=\y$;   however, they are  distinguished  by  $\Psi_{\x, \x}$.
\end{theorem}
\begin{proof} Notice that  the link $\C$  corresponds  to  the  cts--link   $(\C,	 {\1_2})$.  We  denote  by  $\tilde \C$ the link  corresponding to   the  cts--link  $(\C,\{\{1,2\}\})$.  Since  these  links  are  not singular,  we  have  $\Phi_{\x,\y}(\C)=\Psi_{\x,\y}(\C)$ and   $\Phi_{\x,\y}(\tilde\C)=\Psi_{\x,\y}(\tilde\C)$. In particular
$$
\Phi_{\x, \y}(\C)=
\Phi_{\x, \y}(\A) \Phi_{\x, \y}(\B) / (\a \sqrt{\c}) \quad   \text{and} \quad\Phi_{\x, \y}(\tilde \C) =
\Phi_{\x, \y}(\A) \Phi_{\x, \y}(\B) \f /  \sqrt{\c},
$$
where  $\f:= (\u \c-1)/(1-\u)  =\b/\a$, see  \cite{aijuJKTR1}.
Observe  that  the  knots $\D$ and  $\D'$  in Figure \ref{exa1} are isotopic,    both  corresponding to the  connected  sum of  the  knots  $\A$  and  $\B$, so  that   $\Phi_{\x,\y}(\D)=\Phi_{\x,\y}(\D')= \Phi_{\x,\y}(\A)\Phi_{\x,\y}(\B)$.
 Using now  the  desingularization rule IV we  get
$$  \Phi_{\x, \y}(\Sing)=  \x \ \Phi_{\x, \y}(\D)+ \y \Phi_{\x, \y}(\tilde \C),$$
$$
\Phi_{\x, \y}(\Sing')= \x \ \Phi_{\x, \y}(\tilde \C)+  \y \Phi_{\x, \y}(\D').
$$
Therefore,
$$
\Phi_{\x, \y}(\Sing)=\Phi(\A)\Phi_{\x, \y}(\B)(\x + \y  \f /  \sqrt{\c} ),$$  $$
\Phi_{\x, \y}(\Sing')=\Phi(\A)\Phi_{\x, \y}(\B)(\x \ \f/  \sqrt{\c}  +\y ).
$$
 These  values coincide  if  and only if  $\x=\y$.
In fact, $\Phi_{\x, \y}(\Sing)=\Phi_{\x, \y}(\Sing')$ implies $(\x-\y)(1-\f/ \sqrt{\c})=0$; now the equation $(1-\f/ \sqrt{\c})=0$ has  solutions  $\c=1$  and    $\c=\u^{-2}$,   so $\Phi_{\x, \y}$  distinguishes $\Sing$  and  $\Sing'$ if and only if $\x\not= \y$.

Consider now the  polynomial $\Psi_{\x, \y}$.  Using  the  desingularization rule IV',  we  get:
$$
\Psi_{\x, \y}(\Sing)=  \x \ \Psi_{\x, \y}(\D)+ \y \Psi_{\x, \y}(\C),$$
$$
\Psi_{\x, \y}(\Sing')= \x \  \Psi_{\x, \y}(\tilde \C)+  \y \Psi_{\x, \y}(\D').
$$
Therefore,
$$
\Psi_{\x, \y}(\Sing)=\Psi_{\x, \y}(\A)\Psi_{\x, \y}(\B)(\x+ \y/ (\a \sqrt{\c})  ),
$$
$$
\Psi_{\x, \y}(\Sing')=\Psi_{\x, \y}(\A)\Psi_{\x, \y}(\B)(\x \f/\sqrt{\c} +\y ).
$$
   Now, if $\x=\y$,  the  equation   $\Psi_{\x, \y}(\Sing)=\Psi_{\x, \y}(\Sing')$ implies  $\b=1$, hence $\Psi_{\x, \x}$ distinguishes $\Sing$ from $\Sing'$.
\end{proof}

\begin{remark}\rm
  Up to the present we don't have    examples showing that $\Psi_{\x, \y}$ is   able to distinguish pairs of     classical  singular links not distinguished  $\Phi_{\x, \y}$.
\end{remark}

\subsection{Comparison of our invariants with  known invariants}
\mbox{}

\begin{theorem}\label{comparison1}
Let  $\LLL_1$  and  $\LLL_2$  be two non isotopic links  distinguished  by  $\F$    but  not  by  $P$. Then any pair of  singular  links obtained by  adding to $\LLL_r$  ($r=1,2$) a new component making  a singular  crossing and  a  negative  crossing  with  whatever component of $\LLL_r$, is
distinguished  by  $\Phi_{\x,\y}$  and  by  $\Psi_{\x,\y}$   but  not  by  $\I_{PR}$.
\end{theorem}

\begin{example}\label{exa2} \rm   The  cts--links  $(\N_1,\1_4)$  and  $(\N_2,\1_4)$  in Figure \ref{N1N2}  have  one  singularity.  They  are distinguished  by  the  polynomials  $\Phi_{\x,\y}$, $\Psi_{\x,\y}$, but not by  $\Phi_{\x,\y}'$, $\Psi_{\x,\y}'$, nor by $\I_{PR}$. Indeed, by removing  the orange  component,  we  obtain   the  pair $\LLL11n356\{1,0\}$  and  $\LLL11n434\{0,0\}$, distinguished  by  $\F$   but not by  $\F'$, nor by  $P$, see \cite{aica}.
\begin{center}
 \begin{figure}[H]
 \includegraphics{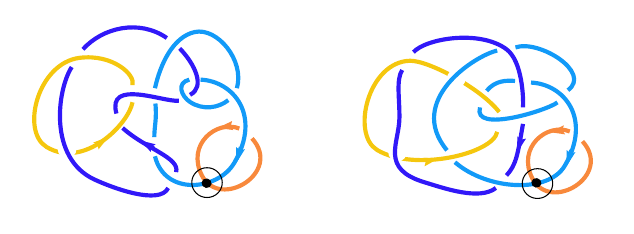}
 \caption{Two links $(\N_1,\1_4)$ and $(\N_2,\1_4)$  distinguished by  the  polynomials  $\Phi_{\x,\y}$, $\Psi_{\x,\y}$, but not by  $\Phi_{\x,\y}'$,   $\Phi_{\x,\y}'$ ,   $\I_{PR}$  .}
\label{N1N2}
 \end{figure}
 \end{center}

\end{example}

\begin{theorem}\label{comparison2}
Let  $\LLL_1$  and  $\LLL_2$  be two non isotopic links  distinguished  by  $\F'$  but  not  by  $P$. Then any pair of  singular  links obtained by  adding to $\LLL_r$  ($r=1,2$) a new component making  a singular  crossing and  a  negative  crossing  with  whatever component of $\LLL_r$, is distinguished  by  $ \Phi_{\x,\y}'$  and  by  $  \Psi_{\x,\y}'$  but  not  by  $\I_{PR}$.
\end{theorem}

\begin{example}\label{exa3} \rm   The  cts--links  $(\M_1,\1_4)$  and  $(\M_2,\1_4)$  in Figure \ref{L1L2}  have  one  singularity.  They  are distinguished  by  the  polynomials  $\Phi_{\x,\y}$, $\Psi_{\x,\y}$, $\Phi_{\x,\y}'$  and  by  $\Psi_{\x,\y}'$, but not  $\I_{PR}$. Indeed, by removing  the orange  component,  we  obtain  the    pair $\LLL10n79\{1,1\}$  and  $\LLL10n95\{1,0\}$, distinguished  by  $\F$ and  by  $\F'$  but not by  $P$, see  respectively  \cite{aica} and \cite{chjukalaIMRN}.

 \begin{center}
 \begin{figure}[H]
 \includegraphics{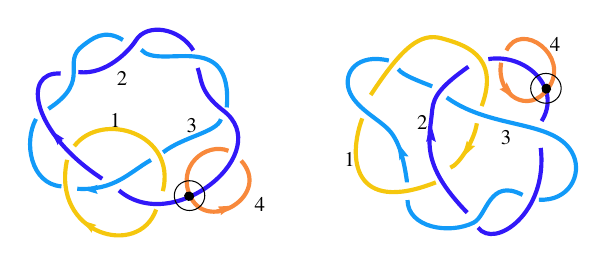}
 \caption{Two links $(\M_1,\1_4)$ and $(\M_2,\1_4)$  distinguished by  the  polynomials  $\Phi_{\x,\y}$, $\Psi_{\x,\y}$, $\Phi_{\x,\y}'$, and $\Phi_{\x,\y}'$  but  not by $\I_{PR}$  .}
\label{L1L2}
 \end{figure}
 \end{center}

\end{example}

\begin{proof}[Proof of  Theorem \ref{comparison1}]  We use  Example  \ref{exa3}  to  illustrate  the  proof.    By the  desingularization skein
  rule  IV,  we  get  for the pair  $(\M_r,\1_4)$, $r=1,2$ (see the pairs A and C   in Figure \ref{L1L2proof}):
$$
  \Phi_{\x,\y}(\M_{r,\times}^{2,4},\1_4)= \x  \  \Phi_{\x,\y} (\M_{r,0}^{2,2},\{\{1\}_1,\{2\}_2,\{3\}_3\}) + \y \    \Phi_{\x,\y}(\M_{r,+}^{2,2},\{\{1\}_1,\{2,4\}_2,\{3\}_3\}).
$$

Now, observe  that the pair    $(\M_{r,+}^{2,2},\{\{1\}_1,\{2,4\}_2,\{3\}_3\})_{r=1,2}$ corresponds to the  pair  (A in Figure \ref{L1L2proof})  of  tied links  $(\LLL_1\tilde \sqcup O, L_2 \tilde\sqcup O)$, where  the  symbol $\tilde \sqcup$  means that  there is  a  tie  between $\LLL$  and the  unknot,    while the pair  $(\M_{r,0}^{2,2},\{\{1\}_1,\{2\}_2,\{3\}_3\})_{r=1,2}$   (C in Figure \ref{L1L2proof}) is  the pair   $(\LLL_1,\LLL_2)$.   Notice  that,  by    Proposition \ref{valuesinv}, the  value of  $\Phi_{\x,\y}$    on  these pairs  is    the  value  of  $\F$,  which distinguishes the  pair  $(\LLL_1,\LLL_2)$.  Observe, moreover,  that  the value of  $\F$ on $\LLL_r\tilde \sqcup O $ is  the  value of $\F$ on $\LLL_r$ by a  coefficient independent from $\LLL_r$;  therefore $\F$ distinguishes  both pairs.  As for $\Psi$,  we  have
$$
  \Psi_{\x,\y}(\M_{r,\times}^{2,4},\1_4)= \x  \   \Psi_{\x,\y} (\M_{r,0}^{2,2},\{\{1\}_1,\{2\}_2,\{3\}_3\}) + \y \    \Psi_{\x,\y}(\M_{r,+}^{2,4},\1_4 ).
$$
  In Figure \ref{L1L2proof}, the pair $(\M_{r,+}^{2,4},\1_4)_{r=1,2}$ is the pair B   and $(\M_{r,0}^{2,2},\{\{1\}_1,\{2\}_2,\{3\}_3\})_{r=1,2}$  is again the pair C, i.e. $(\LLL_1,\LLL_2)$. Also in this  case,  by  Proposition \ref{valuesinv}, the  value of  $\Psi$    on  these pairs  is    the  value  of  $\F$, which distinguishes both pairs.
\end{proof}

 \begin{center}
 \begin{figure}[H]
 \includegraphics{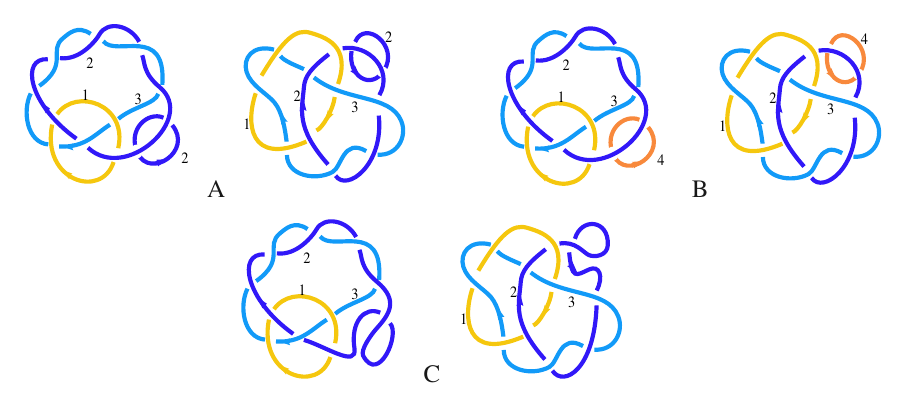}
 \caption{The  links  obtained  by  the  desingularization rules  of  the pair $(\M_1,\1_4)$ and $(\M_2,\1_4)$.  }
\label{L1L2proof}
 \end{figure}
 \end{center}

\begin{proof}[Proof of  Theorem \ref{comparison2}]

For the  values of $\Phi_{\x,\y}'$  and  $\Psi_{\x,\y}'$, the  argument is  exactly the  same  by  using  $\F'$ instead of $\F$.
The  value of  $\I_{PR}$, instead,  is  obtained by  substituting  the  partitions in the  last formula by  the  partitions with a  sole  part, see   Proposition  \ref{valuesinv}.  Thus the  values of  $\F$  and  $\F'$ coincide,  again by    Proposition  \ref{valuesinv},  with the  values of  $P$,  which  does not  distinguish  such pairs.  \end{proof}

\begin{proposition}\label{comparison3}
The pairs  of singular links,  denoted $\C_1$  and  $\C_2$  in  Figure \ref{L3L4}, are  both  distinguished by  $\Phi_{\x,\y}$, $\Psi_{\x,\y}$, $\Phi_{\x,\y}'$   and  by  $\Psi_{\x,\y}'$,   but are not distinguished by  $\I_{PR}$.
\end{proposition}

 \begin{center}
 \begin{figure}[H]
 \includegraphics{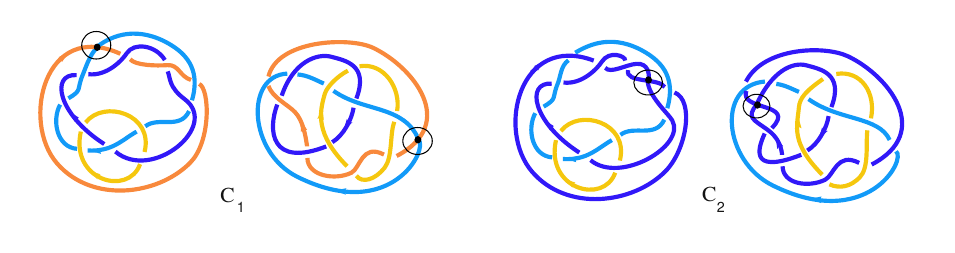}
 \caption{Pairs of links  distinguished by  the  polynomials  $\Phi_{\x,\y}$  and   $\Psi_{\x,\y}$.}
\label{L3L4}
 \end{figure}
 \end{center}

\begin{proof}  Let us  denote by $(\LLL_1,\LLL_2)$  the pair of classical links $\LLL10n79\{1,1\}$  and  $\LLL10n95\{1,0\}$,   and   by  $(\LLL_3, \LLL_4)$ the pair of classical links $\LLL11n325\{1,1\}$  and  $\LLL11n424\{0,0\}$, see  Figure \ref{M1M2}.  All  these  links  have 3  components,  and  both  pairs  are  distinguished  by  the  polynomials  $\F$  and  $\F'$ but not    by  the  $P$.

Consider now  the pair of classical  links  $(\M_1,\M_2)$  in Figure \ref{M1M2}, with  four components.  Also  this  pair is  not distinguished  by    $P$  but is  distinguished by  $\F$  and by  $\F'$,  also  when  two of the four components  belong to  the  same part.

For the pair  $\C_1$  of  singular links $\Sing_1$  and  $\Sing_2$  we  apply  the    desingularization  rule IV  and   we obtain:

$$  \Phi_{\x, \y} (\Sing_1)=  \x \ \Phi_{\x, \y}(\LLL_1,\1_3)  + \y  \Phi_{\x, \y}(\M_1^{3,4}, \{\{3,4\}\} ),    $$
$$  \Phi_{\x, \y} (\Sing_2)=  \x \  \Phi_{\x, \y}(\LLL_2,\1_3) +   \y  \Phi_{\x, \y}(\M_2^{3,4},\{\{3,4\}\} ) .  $$

For the pair  $\C_2$  of  singular links $\Sing_3$  and  $\Sing_4$  we  apply  the    desingularization  rule IV  and  we obtain:

$$  \Phi_{\x, \y} (\Sing_3)=  \x \  \Phi_{\x, \y}(\M_1^{2,4}, \{\{2,4\}\} )+   \y \Phi_{\x, \y}(\LLL_3,\1_3) ,  $$
$$  \Phi_{\x, \y} (\Sing_4)=  \x \ \Phi_{\x, \y}(\M_2^{2,4},\{\{2,4\}\} )  +  \y  \Phi_{\x, \y}(\LLL_4,1_3) .  $$

Now, by     Proposition  \ref{valuesinv},  we have that  $\Phi_{\x, \y}(\LLL_r,\1_3)=\F(\LLL_r)$  and the  value  of $\Phi_{\x, \y}$  on  $(\M_r^{2,4}, \{\{2,4\}\} )$ and  $(\M_r^{2,4}, \{\{3,4\}\} )$ coincides   with that of $\F$.

\begin{center}
 \begin{figure}[H]
 \includegraphics{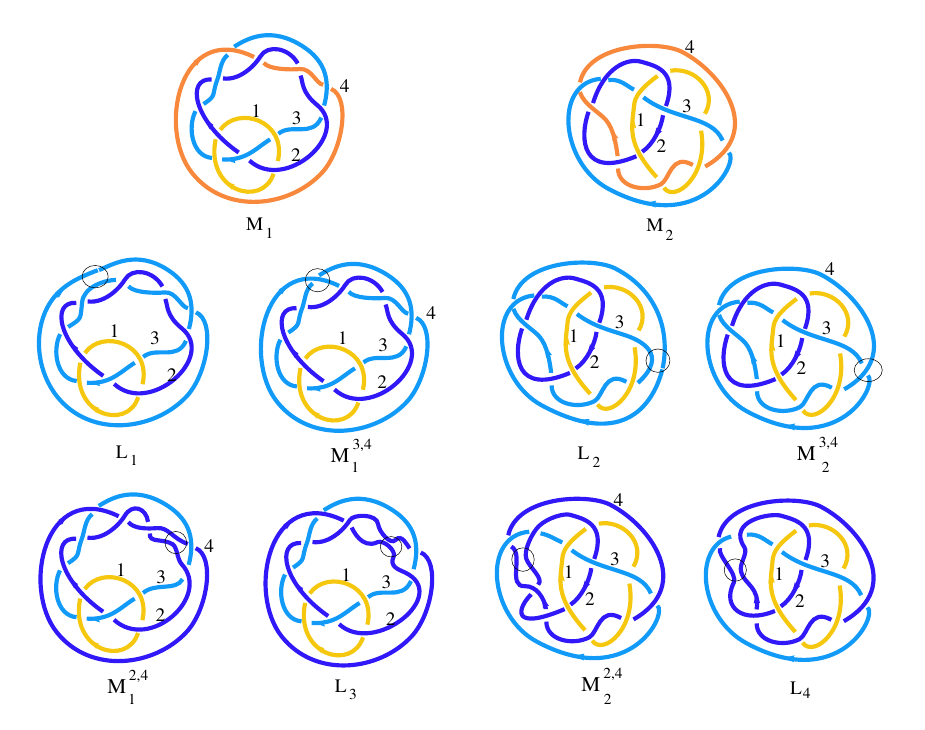}
 \caption{Two  pairs of singular links distinguished by  the  polynomials  $\Phi_{\x, \y}$,  $\Psi_{\x, \y}$ but not by  $\I_{PR}$.}
\label{M1M2}
 \end{figure}
 \end{center}

We do not  write  the  desingularization rules for  $\Phi_{\x,\y}'$,  since  they  give   the  same expressions  by  replacing  $\Phi_{\x,\y}$ with  $\Phi_{\x,\y}'$ and  $\F$ with $\F'$.  So,  $\Phi_{\x,\y}$  and  $\Phi_{\x,\y}'$  distinguish  the pairs $\C_1$ and  $\C_2$ as  a  consequence of  the  fact that  $\F$  and  $\F'$  distinguish the  links obtained  by  the  desingularization.  The  fact  that  $\I_{PR}$  does  not  distinguish these pairs,  follows from  the  fact  that  $P$  does not distinguish the corresponding pairs, see  Proposition  \ref{valuesinv} items (2) and (5).

For the polynomial $\Psi_{\x, \y}$,  the  desingularization rule  applied to   the pair $\C_1$ of  singular links $\Sing_1$  and  $\Sing_2$  gives:

$$  \Psi_{\x, \y} (\Sing_1)=  \x \ \F(\LLL_1)  + \y \F(\M_1, \1_4) ,    $$
$$  \Psi_{\x, \y} (\Sing_2)=  \x \ \F(\LLL_2) +   \y   \F(\M_2,\1_4 ) .  $$
The  same    holds  for  $\Psi'$, replacing  $\F$ with  $\F'$.   Now,   since  the  singularities of     $\Sing_3$ and $\Sing_4$ involve   a  unique component,  the  desingularization rules for  $\Psi$  and $\Psi'$ coincide  with  those for  $\Phi_{\x,\y}$.   Thus,   the  proof  follows  as    that  for $\Phi_{\x,\y}$.

\end{proof}
  Finally,  in   the  proposition below   we  show   the behavior of our invariants  and of  ${\rm I}_{PR}$ on  a  pair of  links with  two  singularities.
\begin{proposition}\label{comparison4}
 The
pair $\C_3$ of  singular  links in Figure \ref{SS2} with  two  singularities  is distinguished by $\Phi_{\x,\y}$,  $\Psi_{\x,\y}$,   $\Phi_{\x,\y}'$, $\Psi_{\x,\y}'$   and  by  $\I_{PR}$.
\end{proposition}

\begin{proof} The desingularizations  of  the two  singular links give    four  pairs of  classical links, some of  them     already  considered  in  Proposition \ref{comparison3}.  However,  the  presence  of other  pairs,  distinguished  by  the classical  polynomials, makes the  original  singular pair distinguished  also  by  $\I_{PR}$.
\end{proof}

\begin{center}
 \begin{figure}[H]
 \includegraphics{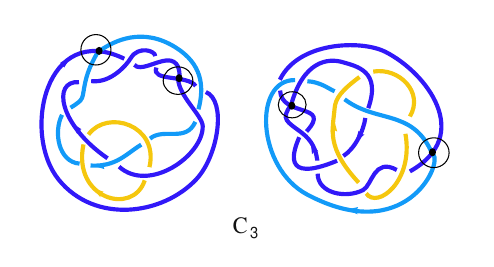}
 \caption{Two singular links distinguished by  the  polynomials  $\Phi_{\x, \y}$, $\Phi'_{\x, \y}$, $\Psi_{\x, \y}$, $\Psi'_{\x, \y}$  and $\I_{PR}$.}
\label{SS2}
 \end{figure}
 \end{center}

\end{document}